\documentclass{article}

\usepackage{microtype}
\usepackage{xspace}
\usepackage{url}
\usepackage[utf8]{inputenc}
\usepackage{amsfonts}
\usepackage{amssymb}
\usepackage[cmex10]{amsmath}
\usepackage{amsthm}
\usepackage{algpseudocode}
\usepackage{grf}
\usepackage[numbers,sort,compress]{natbib}
\usepackage{doi}
\usepackage[hang,centerlast,tight]{subfigure}
\usepackage{hyperref}
\hypersetup{
   unicode=true,
   pdftoolbar=false,
   pdfmenubar=false,
   pdffitwindow=false,
   pdfnewwindow=true,
   pdfkeywords={},
   colorlinks=true,
   linkcolor=black,
   citecolor=black,
   filecolor=black,
   urlcolor=black
}

\usepackage[utf8]{inputenc}
\usepackage[T1]{fontenc}

\usepackage{newfloat}
\usepackage{caption}

\usepackage{color}

\DeclareFloatingEnvironment[fileext=frm,placement={th},name=Algorithm]{algorithm}
\newcommand{\prox}{\text{prox}}
\newcommand{\vect}{\boldsymbol}
\newcommand{\iter}[2]{#1^{( #2 )}}
\newcommand{\viter}[2]{\iter{\vect{#1}}{#2}}
\DeclareMathOperator{\st}{s.t.:}
\DeclareMathOperator{\lev}{lev}
\DeclareMathOperator*{\argmin}{arg min}

\newtheorem{prop}{Proposition}
\newtheorem{theo}{Theorem}
\newtheorem{coro}{Corollary}

\title{Superiorization of Incremental Optimization Algorithms for Statistical Tomographic Image Reconstruction}

\author{Helou, E.S., Zibetti, M.V.W. and Miqueles, E.X.}

\begin{document}
   \maketitle

   \begin{abstract}
      We propose the superiorization of incremental algorithms for tomographic image reconstruction. The resulting methods follow a better path in its way to finding the optimal solution for the maximum likelihood problem in the sense that they are closer to the Pareto optimal curve than the non-superiorized techniques. A new scaled gradient iteration is proposed and three superiorization schemes are evaluated. Theoretical analysis of the methods as well as computational experiments with both synthetic and real data are provided.
   \end{abstract}

   \section{Introduction}

   Tomographic images reconstructed from projection data are important tools in various applications, for example, ranging from medicine to materials science and from geosciences to astronomy. It is therefore useful to develop the techniques that enable good reconstruction from a variety of data acquisition modalities. That means being able to cope with statistical error, poor angular sampling, truncated data, among other difficulties. Sometimes a combination of factors must be dealt with.

   Statistical methods were developed in order to handle poor photon counts, mainly for use in emission tomography modalities. Among those we can mention, in order of appearance, \textsc{em}~\cite{shv82,vsk85}, \textsc{os-em}~\cite{hul94}, \textsc{ramla}~\cite{brp96}, \textsc{bsrem}~\cite{dey01}, \textsc{drama}~\cite{tak03,hed05} and \textsc{saem}~\cite{hcc14}. The original numerical approach to solve the statistical optimization model was the \textsc{em} algorithm, which was deemed too slow, taking many iterations to provide reasonably accurate images; \textsc{os-em} subdivides the data into subsets and processes each of these incrementally in order to achieve an order of magnitude speedup compared to \textsc{em}, but leads to oscillatory behavior when applied to inconsistent data; \textsc{ramla} takes the subset approach to the extreme using a single datum at a time, but prevents oscillation through the use of relaxation in order to ensure convergence; \textsc{bsrem} generalizes \textsc{ramla} by considering more flexible subset divisions and by allowing the use of regularization in the objective function; \textsc{drama} introduced variable relaxation within an iteration cycle aiming at a more even noise contribution from each individual datum in the resulting images; finally, \textsc{saem} further enhance the possibilities by considering a parallelization approach whereby the same algorithm is applied in parallel to subsets of the data and the results are averaged to form the next iterate.

   In the present paper, we consider yet another direction for statistical methods. All of the aforementioned algorithms use a diagonal scaling of the descent direction, which provides desirable convergence characteristics and and helps in maintaining non-negativity of the iterations. This diagonal scaling has the drawback of making convergence analysis more difficult in the case of a non-differentiable objective function, which precludes several useful regularization functions. We, therefore, make use of the superiorization framework in order to analyze algorithms for this case. Our analysis will cover all of the above cited methods, allowing for the use of more interesting regularization functions alongside scaled incremental algorithms for the statistical tomographic reconstruction problem. Moreover, we present a new algorithmic framework which has better theoretical characteristics than the techniques mentioned in the previous paragraph, with similar practical performance.

   We use both the superiorized \textsc{saem} algorithm and the newly presented superiorized approach in order to solve real world tomographic problems arising from synchrotron illuminated radiographic data. Implementation of the techniques was very inhomogeneous with the former being implemented using Matlab and the second running on \textsc{gpu}s, so no direct comparison was intended to be made.

   \subsection{Tomographic Imaging from Projections}

   The most common transmission tomography technique makes use of \textsc{x}-rays, based on the Beer-Lambert law:
   \begin{equation}
      \frac{I_d}{I_e} = e^{-\int_L \mu( \vect x )\mathrm ds},
   \end{equation}
   where $I_d$ and $I_e$ are, respectively, the detected and the emitted \textsc{x}-ray beam intensities, $L$ is the line segment connecting emitter and detector and $\mu : \mathbb R^2 \to \mathbb R_+$ is the non-negative linear attenuation factor. Different attenuation factor values usually correspond to features of interest, for example, the body anatomy in human patients, and therefore its knowledge brings important information for applications. The problem is now to recover the attenuation factor $\mu$ at each point in the plane from the non-invasive measurements of line integrals of $\mu$. While, for simplicity of exposition, we focus the two dimensional case in the present paper, volumetric images can be obtained either by stacking planar tomographic images or by directly extending the approach to 3D datasets.

   We simplify the notation by parameterizing the data by the angle $\theta$ between the normal to the integration path and the horizontal axis and by the distance $t$ of this integration path to the origin of the coordinate system. In doing so, we define the Radon transform $\mathcal R$, which takes functions on the plane to functions on the cylinder as follows:
   \begin{equation}
      \mathcal R[ \mu ]( \theta, t ) := \int_{\mathbb R} \mu\left( t\left(\begin{smallmatrix}\cos\theta\\ \sin\theta\end{smallmatrix}\right) + s \left(\begin{smallmatrix}-\sin\theta\\ \cos\theta\end{smallmatrix}\right)\right) \mathrm ds.
   \end{equation}

   \begin{figure}
      \centering%
      \newcommand{\shepplogan}
{%
   \begin{scope}[line width=0pt]
      \path[color=black!100,draw,fill] (0,0)       ellipse (0.69 and 0.92);
      \path[color=black!20,draw,fill]  (0,-0.0184) ellipse (0.6624 and 0.874);

      \path[color=black!30,draw,fill] (0,0.35)       ellipse (0.21 and 0.25);
      \path[color=black!30,draw,fill] (0,-0.1)       ellipse (0.046 and 0.046);
      \path[color=black!30,draw,fill] (-0.08,-0.605) ellipse (0.046 and 0.023);
      \path[color=black!30,draw,fill] (0,-0.606)     ellipse (0.023 and 0.023);
      \path[color=black!30,draw,fill] (0.06,-0.605)  ellipse (0.023 and 0.046);
      \path[color=black!30,draw,fill] (0.06,-0.605)  ellipse (0.023 and 0.046);

      \path[color=black!0,draw,fill,xshift=0.22\grfxunit,rotate=-18] (0,0) ellipse (0.11 and 0.31);
      \path[color=black!0,draw,fill,xshift=-0.22\grfxunit,rotate=18] (0,0) ellipse (0.16 and 0.41);
      \begin{scope}
         \path[clip,xshift=-0.22\grfxunit,rotate=18] (0,0)    ellipse (0.16 and 0.41);
         \path[color=black!10,draw,fill]             (0,0.35) ellipse (0.21 and 0.25);
         \path[color=black!10,draw,fill]             (0,-0.1) ellipse (0.046 and 0.046);
      \end{scope}
   \end{scope}

   \path[color=black!30,draw,fill]    (0,0.1)  ellipse (0.046 and 0.046);
   \begin{scope}
      \path[clip]                     (0,0.1)  ellipse (0.046 and 0.046);
      \path[color=black!40,draw,fill] (0,0.35) ellipse (0.21 and 0.25);
   \end{scope}
}%
      \input{figradonteta.tex}%
      \def\tlinha{-0.6}%
      \setlength{\grftotalwidth}{0.45\columnwidth}%
      \setlength{\grfticksize}{0.5\grfticksize}%
      \small{\ }\hfill%
      \begin{grfgraphic}{%
         \def\grfxmin{-2.05}\def\grfxmax{1.5}%
         \def\grfymin{-1.5}\def\grfymax{2.05}%
      }%
         \begin{scope}[>=stealth,style=grfaxisstyle,<->]%
            \shepplogan%
            \draw (-1.5,0) -- (1.5,0);%
            \draw (0,-1.5) -- (0,1.5);%
            \foreach \i in {-1,1}%
            {
               \draw[style=grftickstyle,-] (\i\grfxunit,-\grfticksize) -- (\i\grfxunit,\grfticksize);%
               \draw[style=grftickstyle,-] (-\grfticksize,\i\grfyunit) -- (\grfticksize,\i\grfyunit);%
            }
            \begin{scope}[rotate=\teta]
               \draw (-1.5,0) -- (1.5,0);
               \foreach \i in {-1,1}
                  \draw[style=grftickstyle,-] (\i\grfxunit,-\grfticksize) -- (\i\grfxunit,\grfticksize);
               \draw[dashed,dash phase=-0.005\grfyunit] (\tlinha,-1.5) -- (\tlinha,1.5);
               \fill (\tlinha,0) node[anchor=north,inner sep=\grflabelsep] {\scriptsize$t$} circle (0.025cm);
               \draw[-] (\tlinha\grfxunit,0.3em) -| (\tlinha\grfxunit - 0.3em,0pt);
               \fill (\tlinha\grfxunit - 0.15em,0.15em) circle (0.025cm);
            \end{scope}
            \begin{scope}[rotate=\teta,yshift=1.5\grfyunit]
               \draw[line width=0.025cm] plot file {figradon.data};
               \draw[-]  (1.5,0) -- (-1.5,0) node[anchor=north,rotate=\teta,inner sep=\grflabelsep] {\scriptsize$t$};
               \draw[->] (0,0)    -- (0,0.7) node[anchor=west,rotate=\teta,inner sep=\grflabelsep]  {\scriptsize$\mathcal R[ \mu ](\theta,t)$};
            \end{scope}
            \def\rad{0.075}
            \FPupn{\cpt}{0.552285 \rad{} * 90 \teta{} / *}
            \path (\teta:\rad) ++(\teta - 90:\cpt) node (a) {};
            \draw[-] (0,0) -- (\rad,0) .. controls +(0,\cpt) and (a) .. (\teta:\rad) -- cycle;
            \draw[style=grftickstyle,-,rotate=\tetameio,xshift=\rad\grfxunit] (-\grfticksize,0pt) -- (\grfticksize,0pt);
            \path[rotate=\tetameio,xshift=0.3em] (\rad,0) node[anchor=west,inner sep=0pt] {\scriptsize$\theta$};
         \end{scope}
      \end{grfgraphic}\hfill%
      \begin{grfgraphic}[1.125]{%
         \grfyaxis[R]{[]-1;0;1[]}{[]\tiny$-1$;\tiny$0$;\tiny$1$[]}%
         \grfylabel{\footnotesize$t$}%
         \grfxaxis[R]{[]0;1.57;3.14[]}{[]\tiny$0$;\tiny$\frac\pi2$;\tiny$\pi$[]}%
         \grfxlabel{\footnotesize$\theta$}%
         \def\grfxmin{0}%
         \def\grfxmax{3.14}%
         \grfwindow%
      }%
         \node[anchor=north west,inner sep=0pt] at (0,1)
         {\includegraphics[width=3.14\grfxunit,height=2\grfyunit]
         {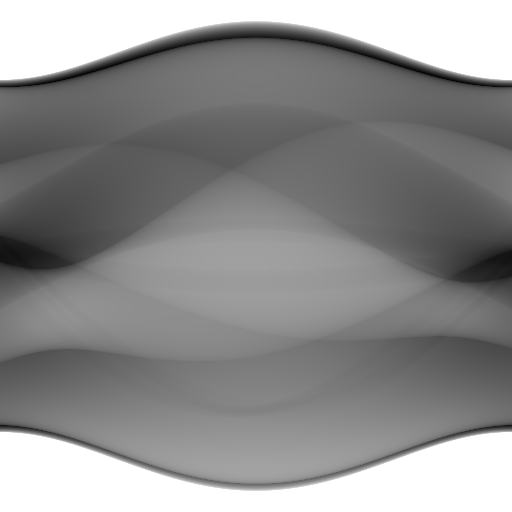}};%
      \end{grfgraphic}\hfill{\ }%
      \caption{Left: meaning of the arguments of the Radon transform. Here, $\theta$ is the angle between the normal to
      the integration path and the horizontal axis, and $t$ is the distance from the line of integration to the origin.
      Right: sinogram, i.e., an image of the Radon transform of the image shown on the left in the $\theta \times t$ coordinate system (adapted from~\cite{hcc14}).}\label{fig:rad}%
   \end{figure}

   Figure~\ref{fig:rad} depicts this definition. The Radon transform is mathematically rich, possessing several useful properties. Among those we can mention, for example, the Fourier slice property, which relates the one-dimensional Fourier transforms of $\mathcal R[\mu]$ along the $t$ coordinate to ``slices'' of the two-dimensional Fourier transform of the image $\mu$. More importantly for us here, the Radon transform turns out to be a compact linear operator which has, in consequence, an ill-conditioned inverse or pseudo-inverse~\cite{nat86}.

   \subsection{Iterative Image Reconstruction}

   While exploring some deep mathematical properties of the Radon transform can lead to useful inversion formulas, sometimes it is better to consider all of the involved physical effects more carefully. For example, the observation that the emission of photons follows a stochastic rule can lead to improved imaging by the introduction of maximum likelihood models whereby the reconstructed image is selected as the one that would give the maximum possible likelihood for the observed data. Whichever the imaging modality is, this approach will give rise to an optimization problem to be solved, therefore bringing the need to use iterative algorithms in order to approximate the solution of the model.

   In every optimization model, whether based on statistical ideas or not, a discretization of the imaged object has to be taken into consideration and, given the finite nature of the data, the Radon transform, or more physically precise stripe-based integral transforms, can be represented by a matrix $R \in \mathbb R^{m \times n}$. In this setting, the image is written as
   \begin{equation*}
      \sum_{j = 1}^n x_j \mu_j,
   \end{equation*}
   where $\{ \mu_1, \mu_2, \dots, \mu_n \}$ is a basis for the space where the imaged object is assumed to be and $x_j$ are the components of vector $\vect x \in \mathbb R^n$. In all of our models, the basis will be the indicator functions of square pixels which also means that the $x_j$ are the actual values of the pixelized image. If taken literally, the problem would then become to find the solution of the linear system of equations
   \begin{equation*}
      R \vect x = \vect b,
   \end{equation*}
   where the coefficients of $R$ are $r_{ij} = \mathcal R[ \mu_j ]( \theta_i, t_i )$ and the components of $\vect b$ are approximations to the Radon transform: $b_i \approx \mathcal R[ \mu ]( \theta_i, t_i )$, where $\mu$ is the desired true image. The system matrix arising in tomographic problems ranges from mildly to severely ill-posed, which brings difficulties because of the experimental error in the collected tomographic data $\vect b$.

   Therefore, instead of directly solving the above linear system of equations, it is usually more profitable to consider a (regularized) maximum likelihood model. For example, in emission tomography the desired image is the intensity of photon emission across the imaged plane, i.e., the number of emitted photons on each square pixel~\cite{vsk85}. In his case, maximizing the likelihood of obtaining the data is equivalent to solve the following optimization problem:
   \begin{equation*}
      \begin{split}
         \min & \quad L_{\textsc{e}}( \vect x ) := \sum_{i = 1}^m\left\{ ( R\vect x )_i - b_i\ln ( R\vect x )_i \right\}\\
         \st  & \quad \vect x \in\mathbb R_+^n,
      \end{split}
   \end{equation*}
   where, in this case, $b_i$ represents the attenuation corrected number of photons detected on bin $i$.

   For the case of transmission tomography, if we denote as $\rho_i$ the number of events detected by detector $i$ during a dark scan (i.e., no \textsc{x}-ray source turned on), $\beta_i$ to be the number of photons detected by detector $i$ during a blank scan and $\alpha_i$ the number of photons detected by detector $i$ during the actual tomographic scan of the object, the maximum likelihood problem is equivalent to~\cite{rak13}:
   \begin{equation*}
      \begin{split}
         \min & \quad L_{\textsc{tr}}( \vect x ) := \sum_{i = 1}^m\left\{ \beta_ie^{-( R\vect x )_i} - \alpha_i \log\left( e^{-(R \vect x )_i} + \rho_i\right) \right\}\\
         \st  & \quad \vect x \in\mathbb R_+^n.
      \end{split}
   \end{equation*}

   Both of the above functions are convex~\cite{hil93} (convexity of $L_{\textsc{tr}}$ requires assumptions on the data which are always satisfied in practice), although not always strictly convex. Non-convex models do exist for the tomographic imaging problem, but due to the usually large dimensions of the problem algorithms for the convex optimization problem are more practical and much more common in practice. We will focus only on convex models in the present paper. In the next section we present the proposed algorithms and the general optimization model they intend to solve.

   \section{Algorithm Description}

   For optimization problems like
   \begin{equation}\label{eq:prob_basic}
      \begin{split}
         \min & \quad f( \vect x )\\
         \st  & \quad \vect x \in \mathbb R_+^n,
      \end{split}
   \end{equation}
   with $f : \mathbb R^n \to \mathbb R$ convex and differentiable, we will consider algorithms of the form
   \begin{equation}\label{eq:alg_opt}
      \viter x{k + 1} = \viter xk - \lambda_kD( \viter xk )\nabla f( \viter xk ) + \viter \epsilon k,
   \end{equation}
   where the diagonal scaling matrix will be given differently according to one of two algorithmic schemes to be described later.
   Furthermore, it is assumed that
   \begin{equation*}
      \lambda_k > 0, \quad \sum_{k = 0}^\infty \lambda_k = \infty.
   \end{equation*}

   For us, the usefulness of allowing the error sequence $\{ \viter \epsilon k \}$ in the analysis will be twofold. First, as already observed before in the literature~\cite{brp96,dey01,hed05,hcc14}, it allows for flexible, from sequential to parallel with many intermediate instances, processing of the data giving rise to fast and stable algorithms. Second, it will allow us to apply the superiorization concept in order to drive the iterates to the optimum through a path with smoother intermediate iterates, as follows. We write an algorithm following~\eqref{eq:alg_opt} in a  compact form as
   \begin{equation}
      \viter x{k + 1} = \mathcal O( \lambda_k, \viter xk ),
   \end{equation}
   and we create a new algorithm following
   \begin{equation}\label{eq:algo_sup}
      \viter x{k + 1} = \tilde{\mathcal O}( \lambda_k, \viter xk ) := \mathcal O( \lambda_k, \viter xk ) + \viter sk,
   \end{equation}
   where we name $\{ \viter sk \}$ the \emph{superiorization sequence}. Algorithm~\eqref{eq:algo_sup} is correspondingly named the \emph{superiorized} version of Algorithm~\eqref{eq:alg_opt}. Notice that if we define $\iter{\tilde{\vect\epsilon}}k := \viter\epsilon k + \viter sk$, we then have that iterations~\eqref{eq:algo_sup} are equivalent to
   \begin{equation*}
      \viter x{k + 1} = \viter xk - \lambda_kD( \viter xk )\nabla f( \viter xk ) + \iter{\tilde{\vect\epsilon}} k.
   \end{equation*}
   That is, the essential approximation characteristics of the method are not ruined by adding a deliberate, appropriately small, perturbation. We shall worry about convergence issues in the next section.

   In the remainder of the present section we describe two options for the operator $\mathcal O$, one of which is an existing algorithm while the second is our new proposal. Furthermore, we also describe more precisely the two alternatives for the superiorization sequence we will try in the numerical section.

   \subsection{String-Averaging Expectation Maximization}\label{saem-sup}

   The String-Averaging Expectation Maximization (\textsc{saem}) algorithm was introduced in~\cite{hcc14} and applied to a maximum likelihood problem in tomographic reconstruction. Convergence analysis contained in~\cite{hcc14} embraces somewhat general smooth convex cost functions, and therefore we will keep the acronym \textsc{saem} even when not necessarily referring to applications to likelihood maximization problems.

   In order to describe \textsc{saem}, we assume that the objective function can be written as:
   \begin{equation*}
      f( \vect x ) = \sum_{i = 1}^pf_i( \vect x ),
   \end{equation*}
   where $f : \mathbb R^n \to \mathbb R$ is supposed to be a convex function and $f_i : \mathbb R^n \to \mathbb R$ are sufficiently smooth functions.

   We then assume that the sequence $\{ 1, 2, \dots, p \}$ is split in $s$ (ordered) subsets, $\{ S_1, S_2, \dots, S_s \}$, each of which is called a string, satisfying
   \begin{equation*}
      \bigcup_{l = 1}^s S_l = \{ 1, 2, \dots, p \}\quad\text{and}\quad S_l \cap S_\ell = \emptyset\quad\forall l \neq \ell.
   \end{equation*}
   We make each of the subsets explicit through the following notation for the elements of each string:
   \begin{equation*}
      S_l := \{ \iota^l_1, \iota^l_2, \dots, \iota^l_{\nu( l )} \}.
   \end{equation*}
   That is, $\iota^l_i$ is the $i^{\text{th}}$ element of string $S_l$ and $\nu( l )$ is the cardinality of string $S_l$.

   Now, for each string we iterate a scaled incremental gradient for the function $F_l := \sum_{i \in S_l} f_i$. Let us define the string operators $\mathcal S_l : \mathbb R \times \mathbb R^n \to \mathbb R^n$ for this purpose:
   \begin{equation*}
      \mathcal S_l( \lambda, \vect x ) := \viter y{\nu( l )},
   \end{equation*}
   where $\lambda$ is a positive scalar, $\vect x \in \mathbb R^n$ and the vectors $\viter y{i} \in \mathbb R^n$, $i \in \{ 1, 2, \dots, \nu( l ) \}$ are recursively computed from the identity below:
   \begin{equation*}
      \viter y{i} = \viter y{i - 1} - \lambda D( \viter y{i - 1} )\nabla f_{\iota^l_i}( \viter y{i - 1} ), \quad i \in \{ 1, 2, \dots, \nu( l ) \}.
   \end{equation*}
   Recursion starts from $\viter y0 = \vect x$, and the scaling matrix is given by
   \begin{equation}\label{eq:diagmatrix}
      D( \vect x ) := \begin{pmatrix}
         \displaystyle\frac{x_1}{p_1} & & & &\\
          & \displaystyle\frac{x_2}{p_2} & &\\
          & & \ddots &\\
          & & & \displaystyle\frac{x_n}{p_n}
             \end{pmatrix},
   \end{equation}
   where $p_j > 0$ and $x_1, x_2, \dots, x_n$ are the components of $\vect x$.

   Each of these operators is no more than a block-\textsc{ramla} (\textsc{bramla}) iteration over a specific data subset~\cite{brp96,dey01}. A \textsc{bramla} iteration is a generalization of a \textsc{ramla} iteration whereby blocks of data can be used at each subiteration. Under mild smoothness assumptions on the gradients $\nabla f_i$, it is possible to show that
   \begin{equation*}
      \mathcal S_l( \lambda, \vect x ) = \vect x - \lambda D( \vect x )\nabla F_l( \vect x ) + O( \lambda^2 ).
   \end{equation*}
   See, e.g., \cite[Proposition~5.3]{hcc14}. Therefore, by setting weights $\{ \omega_1, \omega_2, \dots, \omega_s \}$ such that $\sum_{l = 1}^s\omega_l = 1$ and $\omega_l \geq 0$, we can finally define the iterative procedure:
   \begin{equation*}
      \viter x{k + 1} = \sum_{l = 1}^s\omega_l \mathcal S_l( \viter xk, \lambda_k ).
   \end{equation*}
   For which, then, there holds
   \begin{equation*}
      \viter x{k + 1} = \viter xk - \lambda_k\nabla \tilde f( \vect x_k ) + O( \lambda_k^2 ),
   \end{equation*}
   where $\tilde f := \sum_{l = 1}^s\omega_l F_l$ is a weighted version of the original objective function $f = \sum_{l = 1}^sF_l$.

   A key, albeit very simple, observation here is that the convergence analysis naturally absorbs an error term, which makes it attractive for the superiorization approach because we can add suitable perturbations to this error term and still maintain a convergent algorithm. In order to do that, the magnitude of the perturbation must be appropriately controlled, which we will analyze theoretically in depth in the next section. The remaining of the present section presents an option to \textsc{saem} and then discusses the three forms of the perturbation $\{ \viter sk \}$ we consider, which we recall that we have named \emph{superiorization sequences}.

   \subsection{New Superiorized Scaled Gradient Algorithm}

   In the present paper, we not only introduce superiorization of well known maximum likelihood algorithms, but we also present a different method with improved theoretical characteristics when compared to \textsc{ramla} and its relatives, while maintaining the good practical performance of these techniques. In order to describe the new algorithmic framework, there is a key idea: 
   that the diagonal scaling matrix cannot have components vanishing when the corresponding entry in the gradient is positive, in order to avoid the possibility of convergence to a non-optimal point when the objective function is convex but not strictly convex. Of course the precise way of interpreting and implementing this statement is very important to the final computational results, and we describe it in details below.

   In the following description we use the same notation for the strings given in Section~\ref{saem-sup}. The concrete algorithm that we propose uses strings similarly to \textsc{saem}, but is stabilized in a sense that it actually approximates a better-scaled iteration. We name it \textsc{ssaem} from Stabilized \textsc{saem}, and the iterations are as follows:
   \begin{equation*}
      \tilde{\mathcal S}_l( \lambda_k, \viter xk ) := \viter y{k, l, \nu( l )},
   \end{equation*}
   where
   \begin{equation*}
      \viter y{k, l, i} = \viter y{k, l, i - 1} - \lambda_k \overline D( \viter y{k, l, i - 1} )\nabla f_{\iota^l_i}( \viter y{k, l, i - 1} ), \quad i \in \{ 1, 2, \dots, \nu( l ) \},
   \end{equation*}
   and
   \begin{equation*}
      \viter y{k, 0} = \viter xk.
   \end{equation*}
   The diagonal scaling matrix $\overline D$ is given componentwise, for $\vect x \in \mathbb R_+^n$, by
   \begin{equation*}
      \overline D( \vect x )_{jj} := \begin{cases}
                     \frac{x_j}{p_j} & \text{if}\quad x_j > \tau\\
                     \frac{\tau}{p_j} & \text{otherwise,}
                 \end{cases}
   \end{equation*}
   where each $p_j$, $j \in \{ 1, 2, \dots, n \}$ is positive.

   The next iterate is computed in two steps. First the averaging:
   \begin{equation*}
      \viter{\tilde x}{k + 1} = \sum_{l = 1}^s\omega_l \tilde{\mathcal S}_l( \lambda_k, \viter xk ).
   \end{equation*}
   and then a componentwise correction:
   \begin{equation*}
      \iter x{k + 1}_j =
         \begin{cases}
             \iter xk_j + \frac{\iter xk_j}\tau\left( \iter{\tilde x}{k + 1}_j - \iter xk_j \right) & \text{if}\quad \iter xk_j \leq \tau \quad \text{and}\quad \iter{\tilde x}{k + 1}_j < \iter xk_j\\
            \iter{\tilde x}{k + 1}_j & \text{otherwise.}\\
         \end{cases}
   \end{equation*}

   We discuss some useful facts about these iterations in preparation for the complete convergence analysis to be provided in the next section.
   \begin{prop}\label{prop:stab_string_property}
      Fix $l \in \{ 1, 2, \dots, s \}$. Suppose that $\iter xk_j \geq 0$ for all $j \in \{ 1, 2, \dots, n \}$, and that, for every $i \in \{ 1, 2, \dots, \nu( l ) \}$, $\{ \nabla f_{\iota^l_i}( \viter y{k, l, i - 1} ) \}_{k \in \mathbb N}$ is bounded, $\nabla f_{\iota^l_i}$ is Lipschitz in
      \begin{equation*}
       C_{l, i} := \text{the closure of the convex hull of } \{ \viter xk, \viter y{k, l, i - 1} \}_{k \in \mathbb N},
      \end{equation*}
      and that $C_{l, i}$ is in turn bounded. Then:
      \begin{equation*}
         \tilde{\mathcal S}_l( \lambda_k, \viter xk ) = \viter xk - \lambda_k\overline D( \viter xk )\left( \nabla F_l( \viter xk ) + O( \lambda_k ) \right).
      \end{equation*}
   \end{prop}
   \begin{proof}
      Notice that the hypothesis ensure that each function $\overline D( \cdot )\nabla f_{\iota^l_i}( \cdot )$ is Lipschitz, say with Lipschitz constant $M$, in $C_i$ and thus, if we define
      \begin{equation*}
         \begin{split}
            \viter\epsilon{k, l} & {}:= \overline D( \viter xk )\nabla F_l( \viter xk ) - \sum_{i = 1}^{\nu( l )}\overline D( \viter y{k, l, i - 1} )\nabla f_{\iota^l_i}( \viter y{k, l, i - 1} )\\
            & {}= \sum_{i = 1}^{\nu( l )}\left\{ \overline D( \viter xk )\nabla f_{\iota^l_i}( \viter xk ) - \overline D( \viter y{k, l, i - 1} )\nabla f_{\iota^l_i}( \viter y{k, l, i - 1} ) \right\},
         \end{split}
      \end{equation*}
      we then have
      \begin{equation}\label{eq:eps_Lips}
         \begin{split}
            \| \viter\epsilon{k, l} \| & {}\leq \sum_{i = 1}^{\nu( l )}\left\| \overline D( \viter xk )\nabla f_{\iota^l_i}( \viter xk ) - \overline D( \viter y{k, l, i - 1} )\nabla f_{\iota^l_i}( \viter y{k, l, i - 1} ) \right\|\\
            & {}\leq M\sum_{i = 1}^{\nu( l )}\| \viter xk - \viter y{k, l, i - 1} \|.
         \end{split}
      \end{equation}
      Now, we observe that the boundedness assumptions imply that there is a real number $N$ large enough such that
      \begin{equation*}
         \| \viter xk - \viter y{k, l, i - 1} \| \leq (i - 1)\lambda_kN.
      \end{equation*}
      Therefore, using this inequality in~\eqref{eq:eps_Lips} we conclude that
      \begin{equation}\label{eq:eps_bound}
         \| \viter\epsilon{k, l} \| \leq \lambda_kNM\frac{\nu( l )( \nu( l ) - 1 )}2.
      \end{equation}

      Thus, from the definition of the method, the definition of $\viter\epsilon{k, l}$ and~\eqref{eq:eps_bound}:
      \begin{equation*}
         \begin{split}
            \tilde{\mathcal S}_l( \lambda_k, \viter xk ) & {}= \viter xk - \lambda_k\sum_{i = 1}^{\nu( l )}\overline D( \viter y{k, l, i - 1} )\nabla f_{\iota^l_i}( \viter y{k, l, i - 1} )\\
            & {}= \viter xk - \lambda_k \overline D( \viter xk )\nabla F_l( \viter xk ) + \lambda_k\viter\epsilon{k, l}\\
            & {}= \viter xk - \lambda_k \overline D( \viter xk )\nabla F_l( \viter xk ) + O( \lambda_k^2 ).
         \end{split}
      \end{equation*}
      Finally, because the diagonal elements of diagonal matrix $\overline D( \viter xk )$ are positive and bounded from below by construction, the claim follows.
   \end{proof}

   Let us introduce the simplifying notation
   \begin{equation*}
      \mathcal I := \bigl\{ ( l, i ) : l \in \{ 1, 2, \dots, s \}\quad\text{and}\quad i \in \{ 1, 2, \dots, \nu( l ) \} \bigr\}.
   \end{equation*}
   \begin{coro}\label{coro:non_norm_iter}
      Suppose $\iter xk_j \geq 0$ for all $j \in \{ 1, 2, \dots, n \}$, and that, for every $( l, i ) \in \mathcal I$, $\{ \nabla f_{\iota^l_i}( \viter y{k, l, i - 1} ) \}_{k \in \mathbb N}$ is bounded, $\nabla f_{\iota^l_i}$ is Lipschitz in
      \begin{equation*}
       C_{l, i} := \text{the closure of the convex hull of } \{ \viter xk, \viter y{k, l, i - 1} \}_{k \in \mathbb N},
      \end{equation*}
      and that $C_{l, i}$ is in turn bounded. Then:
      \begin{equation*}
         \viter{\tilde x}{k + 1} = \viter xk - \lambda_k\overline D( \viter xk )\left( \nabla \tilde f( \viter xk ) + O( \lambda_k ) \right).
      \end{equation*}
   \end{coro}
   \begin{proof}
      Use the definition of $\viter{\tilde x}{k + 1}$ and Proposition~\ref{prop:stab_string_property} in order to get
      \begin{equation*}
         \begin{split}
            \viter{\tilde x}{k + 1} & {}= \sum_{l = 1}^s\omega_l \tilde{\mathcal S}_l( \lambda_k, \viter xk ) = \viter xk - \lambda_k\sum_{l = 1}^s\omega_l\overline D( \viter xk )\left( \nabla F_l( \viter xk ) + O( \lambda_k ) \right)\\
            & {}= \viter xk - \lambda_k\overline D( \viter xk )\left( \sum_{l = 1}^s\omega_l\nabla F_l( \viter xk ) + \sum_{l = 1}^s\omega_lO( \lambda_k ) \right),
         \end{split}
      \end{equation*}
      from where the claim follows.
   \end{proof}

   Corollary~\ref{coro:non_norm_iter} actually says that the iterations before renormalization satisfy
   \begin{equation*}
      \viter{\tilde x}{k + 1} = \viter xk - \lambda_k\overline D( \viter xk )\viter gk,
   \end{equation*}
   where
   \begin{equation*}
      \viter gk = \nabla \tilde f( \viter xk ) + O( \lambda_k ).
   \end{equation*}
   Thus, given the fact that $\overline D( \viter xk )_{jj} \geq 0$, and by the way the normalization is done, we have
   \begin{equation*}
      \viter x{k + 1} = \viter xk - \lambda_k\tilde D( \viter xk, \viter gk )\viter gk,
   \end{equation*}
   where the diagonal scaling matrix is given by
   \begin{equation*}
      \tilde D( \vect x, \vect g )_{jj} = \begin{cases}
                                                \frac{\tau}{p_j} & \text{if}\quad x_j \leq \tau\quad\text{and}\quad g_j \leq 0\\
                                                \frac{x_j}{p_j} & \text{otherwise.}
                                              \end{cases}
   \end{equation*}

   \subsection{Superiorization Sequences}

   The main advantage of a superiorization technique is that there is a lot of freedom for the sequence of perturbations to be added to the traditional optimization method. It is appealing to obtain a convergent method for a maximum likelihood solution, but, in practice, when iterated to full convergence, the maximum likelihood solution in emission or low count transmission tomography is degraded by noise. Therefore, and given the high computational cost of iterative techniques, one common way of smoothing the results is to stop the algorithm prematurely, having started it from a very smooth initial image.

   While reasonably successful, the early stopping strategy cannot provide images as good as explicit regularization of the objective function in a model such as
   \begin{equation}\label{eq:regprob}
      \begin{split}
         \min & \quad f( \vect x ) + \zeta r( \vect x )\\
         \st  & \quad \vect x \in \mathbb R_+^n,
      \end{split}
   \end{equation}
   where $\zeta > 0$ is the \emph{regularization parameter}, and $r : \mathbb R^n \to \mathbb R$ is a function that penalizes roughness, or otherwise undesirable features, in the image. More precisely, by denoting $\vect x_\zeta$ the optimizer of problem~\eqref{eq:regprob}, it is possible to see that
   \begin{equation*}
      r( \vect x_\zeta ) = \min \{ r( \vect x ) : \vect x \in \mathbb R_+^n, f( \vect x ) = f( \vect x_\zeta ) \}.
   \end{equation*}

   That is, explicit regularization in the optimization model will provide the smoothest possible image among those with the same fitness to data. Although this is a desirable feature, there remains the problem of selecting a proper regularization parameter $\zeta$ under a reasonable computational effort. This task would demand the solution of problem~\eqref{eq:regprob} for several tentative values of the regularization parameter, which would therefore lead to a heavy computational overhead to the reconstruction process.

   We claim that the superiorization framework can be useful in this setting, because by a proper selection of the regularization sequence used to perturb the unregularized method we can obtain iterates that are closer to the optimal $r$ versus $f$ curve traced by the explictly regularized optimization method. This could potentially lead to improved results, for example, when consistency-based early stopping strategies are used.

   For illustrative purposes, we will focus on $r = TV$, the Total Variation functional:
   \begin{equation*}
      TV( \vect x ) := \sum_{i = 1}^{\sqrt n}\sum_{j = 1}^{\sqrt n}\sqrt{ ( x_{i, j} - x_{i - 1, j} )^2 + ( x_{i, j} - x_{i, j - 1} )^2 },
   \end{equation*}
   where we have used the simplifying assumption that the reconstructed images are square (which can be readily dropped), the lexicographic ordering of elements, and a periodic boundary condition:
   \begin{equation*}
      x_{0, j} := x_{n, j} \quad\text{and}\quad x_{i, 0} := x_{i, n}.
   \end{equation*}

   \subsubsection{Standard superiorization procedure}

   \begin{algorithm}[t]
      \begin{algorithmic}[1] 
         \item
         {\textbf{set} $k \leftarrow 0$}  \\
         {\textbf{set} $\boldsymbol{x}^{\left(0\right)}$ as initial solution} \\
         {\textbf{while }$ f( \viter xk)  > \varepsilon$   } \\
         {~~~~\textbf{set $\viter x{k + 1 / 2} \leftarrow \mathcal O( \lambda_k, \viter xk ) $}} \\
         {\label{sup-part}~~~~\textbf{set} $\viter x{k+1} \leftarrow    \vect R_{r} (\viter x{k + 1 / 2} )$}\\
         {~~~~\textbf{set $k\leftarrow k+1$}} \\
         {\textbf{return} $\boldsymbol{x}^{\left(k\right)}$}
      \end{algorithmic}
      \caption{General Superiorized Version of Algorithm $\mathcal O( \viter xk ) $}
      \label{gen-super}
   \end{algorithm}

      The general superiorized algorithm \eqref{eq:algo_sup} is exposed in Algorithm~\ref{gen-super}, where it is implicit in the superiorization step $\vect R_{r} (\viter x{k + 1 / 2})$  the procedure for obtaining the superiorized sequence $\viter sk = \vect R_{r} (\viter x{k + 1 / 2}) - \viter x{k + 1 / 2}$. It is expected that the superiorization step improves the solution moving the current iterate towards the minimum of $r( \vect x )$. Also, note that one could choose $\viter xk$ or $\viter x{k + 1 / 2}$ as the reconstructed image, it is just a matter of visual preference\footnote{In \cite{Garduno2014}, the reconstructed image is $\viter x{k + 1 / 2}$, and it used in the stopping criteria $ f( \viter x{k + 1 / 2})  > \varepsilon$. }. Algorithm~\ref{super-seq} shows the method proposed in \cite{Garduno2014} to produce a sequence which it is called here the \emph{standard} superiorization sequence. In Algorithm~\ref{super-seq}, $\beta_0 > 0$ and $\alpha \in (0,1)$, and $\beta_0\alpha^k$ is the starting trial stepsize at iteration $k$, which is repeatedly reduced by a factor of $\alpha$ until the decrease criterion is satisfied.

   \begin{algorithm}[t]
      \begin{algorithmic}[1]
         \item
         {~~\textbf{set} $\ell \leftarrow k$} \\
         {~~\textbf{set} $n \leftarrow 0$} \\
         {~~\textbf{set} $\viter bn \leftarrow \viter x{k + 1 / 2}$}\\
         {~~\textbf{while $n<N$}}\\
         {\label{Line:NAV}~~~~~~\textbf{set} $ \viter v{n}$ to be a \textbf{\textit{nonascending}} vector for $r$ at $\viter bn$} \\
         {~~~~~~\textbf{set} $loop \leftarrow true$} \\
         {~~~~~~\textbf{while} $loop$}\\
         {~~~~~~~~~~\textbf{set $\ell \leftarrow\ell+1$}}\\
         {~~~~~~~~~~\textbf{set} $\beta_n \leftarrow\beta_{0}\times\alpha^{\ell}$}\\
         {~~~~~~~~~~\textbf{set} $\vect z \leftarrow \viter b{n} + \beta_n \viter v{n} $} \\
         {~~~~~~~~~~\textbf{if $r\left( \vect z \right) \leq r\left( \viter x{k + 1 / 2} \right)$ then}}\\
         {~~~~~~~~~~~~~~\textbf{set $n \leftarrow n+1$}}\\
         {~~~~~~~~~~~~~~\textbf{set $ \viter b{n} \leftarrow  \vect z $}}\\
         {\label{Line:NAV2}~~~~~~~~~~~~~~\textbf{set  }\emph{$loop\leftarrow false$}}\\
         {\textbf{return} $ \viter x{k + 1 } \leftarrow \viter bn $} 
      \end{algorithmic}
      \caption{Standard superiorization procedure to obtain $\viter x{k + 1 }$ from point $\viter x{k + 1 / 2}$ }
      \label{super-seq}
   \end{algorithm}

   \subsubsection{Projected subgradient iteration}

   A well known algorithm used for the minimization of convex, non-differentiable, functions under constraints is the projected subgradient iteration, where the iteration operator has the form
   \begin{equation*}
      \mathcal S( \vect x, \lambda ) := \vect x - \lambda \tilde\nabla r( \vect x ).
   \end{equation*}
   Here, $r : \mathbb R^n \to \mathbb R$ is the function to be minimized, $\tilde\nabla r( \vect x ) \in \partial r( \vect x )$ and
   \begin{equation*}
      \partial r( \vect x ) := \{ \vect v \in \mathbb R^n : r( \vect x ) + \nabla \vect v^T( \vect y - \vect x ) \leq r( \vect y ) \}
   \end{equation*}
   is the subdifferential of $r$ at $\vect x$. In our application, the subgradient-based superiorization operator is given as repeated application of a subgradient step under a diminishing stepsize scheme, followed by projection onto the non-negative orthant:
   \begin{equation}\label{eq:proj_subgrad}
      \begin{split}
         \viter y0 & {}= \vect x\\
         \viter yi & {}= \viter y{i - 1} - \frac \gamma i\tilde\nabla r( \viter y{i - 1} ), \quad i = 1, 2, \dots, N\\
         R_r( \vect x ) & {}= \mathfrak P_{\mathbb R_+^n}( \viter yN ),
      \end{split}
   \end{equation}
   where $\mathfrak P_{X}$ is, for a non-empty convex and closed set $X$, the projection:
   \begin{equation*}
      \mathfrak P_{X}( \vect x ) = \argmin_{\vect y \in X}\| \vect x - \vect y \|.
   \end{equation*}
   We will omit the starting stepsize $\gamma$ from the notation, but it is implicitly assumed that each application of the technique has a well defined value for the parameter $\gamma > 0$. When considered as part of a iterative scheme, a sequence $\{ \gamma_k \}$ is assumed.

   The subgradient $\vect t \in \partial TV( \vect x )$ we have used in this superiorization scheme and as a nonascending direction for the standard superiorization procedure is given componentwise by:
   \begin{multline*}
      t_{i,j} = \frac{2x_{i, j} - x_{i, j - 1} - x_{i - 1, j}}{\sqrt{( x_{i, j} - x_{i, j - 1} )^2 + ( x_{i, j} - x_{i - 1, j} )^2}} +{}\\
      \frac{x_{i, j} - x_{i, j + 1}}{\sqrt{( x_{i, j + 1} - x_{i, j} )^2 + ( x_{i, j + 1} - x_{i - 1, j + 1} )^2}} +{}\\
      \frac{x_{i, j} - x_{i + 1, j}}{\sqrt{( x_{i + 1, j} - x_{i, j} )^2 + ( x_{i + 1, j} - x_{i + 1, j - 1} )^2}},
   \end{multline*}
   where the fractions that have denominator equal to $0$ are ignored in the summation.

   \subsubsection{Fast Gradient Projection}

   In \cite{bet09b}, Beck and Teboulle proposed an accelerated method to compute the proximal operator for positively constrained Total Variation given by:
   \begin{equation}
      \prox_{TV,\mathbb R_+^n}( \vect b ) := \arg \min_{\vect x \in \mathbb R_+^n }  \|  \vect x - \vect b   \|^2 +  \gamma TV( \vect x ).
      \label{eq:TV-PROX}
   \end{equation}
   The approach in \cite{bet09b} is based on the dual formulation proposed by Chambolle \cite{Chambolle2004}, where $TV$ is written as:
   \begin{equation}
      TV( \vect x ) := \max_{ \vect p, \vect q \in \mathcal{P} }  T(\vect x, \vect p, \vect q),
      \label{eq:TV-dual-1}
   \end{equation}
   where:
   \begin{equation*}
       T(\vect x, \vect p, \vect q) := \sum_{i = 1}^{\sqrt n}\sum_{j = 1}^{\sqrt n}  p_{i, j}( x_{i, j} - x_{i - 1, j} ) +  q_{i, j}( x_{i, j} - x_{i, j - 1} ) 
      \label{eq:TV-dual-2}
   \end{equation*}
   and $\mathcal{P}$ is a set in which $ \vect p$ and  $\vect q$ satisfy:
   \begin{equation*}
            p_{i, j}^2 + q_{i, j}^2 \leq 1 \quad  \text{for} \quad  1 \leq i , j \leq {\sqrt n}.
         \label{eq:TV-dual-3}
   \end{equation*}
   This way, \eqref{eq:TV-PROX} can be written as:
   \begin{equation}
         \prox_{TV,\mathbb R_+^n} ( \vect b ) := \arg \min_{\vect x \in \mathbb R_+^n } \max_{ \vect p, \vect q \in \mathcal{P} }  \|  \vect x - \vect b   \|^2  + \gamma  (\vect p, \vect q)^T \mathcal{D} \vect x,
        \label{eq:DTV-PROX}
   \end{equation}
   where $\mathcal D$ is the linear operator that takes $\vect x$ and returns the pair $\vect u$ and $\vect v$ such that
   \begin{equation*}
      u_{i, j} = x_{i, j} - x_{i - 1, j} \quad \text{and} \quad v_{i, j} = x_{i, j} - x_{i, j - 1}.
   \end{equation*}

   Equation \eqref{eq:DTV-PROX} is solved with gradient projection algorithm in \cite{Chambolle2004}, and with a fast gradient projection (\textsc{fgp}) in \cite{bet09b}, the approach we are utilizing in this paper to produce the superiorized sequence. Notice that applying \eqref{eq:DTV-PROX} just after the main algorithm: 
   \begin{equation*}
         \viter x{k + 1}  =   \prox_{TV,\mathbb R_+^n} ( \mathcal O( \lambda_k, \viter xk )) = \mathcal O( \lambda_k, \viter xk )  + \viter sk,
   \end{equation*}
   produces the superiorized sequence. Again, we avoid cluttering the notation by omitting the positive parameter $\gamma$.

   \section{Theoretical Convergence Analysis}

   \subsection{SAEM-SUP}

   While the concrete form of our algorithm is based on the perturbed \textsc{saem} as described in the previous section, we will focus in a more general diagonally scaled gradient descent algorithm of the form
   \begin{equation}
      \viter x{k + 1} = \viter xk - \lambda_kD( \viter xk )\nabla f( \viter xk ) + o( \lambda_k ),
   \end{equation}
   where the diagonal scaling matrix $D( \vect x )$ is given in~\eqref{eq:diagmatrix} and
   \begin{equation*}
      \lim_{\lambda \to 0}\frac{\| o( \lambda ) \|}{\lambda} = 0.
   \end{equation*}

   It will be useful to separate the algorithm in two steps, as with Algorithm~\ref{gen-super}:
   \begin{equation}\label{eq:algo_2step}
      \begin{split}
         \viter x{k + 1 / 2} &= \mathcal O( \lambda_k, \viter xk )\\
         \viter x{k + 1} &= \viter x{k + 1 / 2}  + \viter sk.
      \end{split}
   \end{equation}
   We will assume in this subsection that the algorithm is able to maintain positivity for both subsequences $\{ \viter x{k + 1 / 2} \}$ and $\{ \viter xk \}$. This will be the case for our techniques and is easy to achieve when using the diagonal scaling matrix $D( \vect x )$ with the gradient descent direction in the first step and a projected subgradient or constrained proximal step in the second part of the algorithm.


   \begin{prop}\label{prop:conv_weak}
      Consider the iterative procedure described in~\eqref{eq:algo_2step}. Assume the sequences $\{ \viter xk \}$ and $\{ \viter x{k + 1 / 2} \}$ are bounded and positive, and that $\{ \nabla f( \viter xk ) \}$ and $\{ \nabla f( \viter x{k + 1 / 2} ) \}$ are bounded. Suppose also that the first step operator satisfies
      \begin{equation}\label{eq:ass_step1}
         \mathcal O( \lambda, \vect x ) = \vect x - \lambda D( \vect x )\nabla f( \vect x ) + o( \lambda ),
      \end{equation}
      and that
      \begin{equation*}
         \lambda_k \to 0^+,\quad \sum_{k = 0}^\infty \lambda_k = \infty,\quad\text{and}\quad\frac{\| \viter sk \|}{\lambda_k} \to 0.
      \end{equation*}
      Then either there is a subsequence $\{ \viter x{l_k} \}$ satisfying $D( \viter x{l_k} )\nabla f( \viter x{l_k} ) \to \vect 0$ or we have $f( \viter xk ) \to -\infty$.

   \end{prop}
   \begin{proof}
      Let us assume that there is no subsequence $\{ \viter x{l_k} \}$ satisfying the limit $D( \viter x{l_k} )\nabla f( \viter x{l_k} ) \to \vect 0$. Notice that this statement is equivalent to saying that there is no subsequence such that $D( \viter x{l_k} )^{1 / 2}\nabla f( \viter x{l_k} ) \to \vect 0$. In this case, there is $\epsilon > 0$ such that $\| D( \viter x{k} )^{1 / 2}\nabla f( \viter x{k} ) \|^2 \geq \epsilon$ for every $k$.

      Therefore, we have
      \begin{equation*}
         \begin{split}
            f( \viter x{k + 1} ) & {}= f\bigl( \viter xk - \lambda_kD( \viter xk )\nabla f( \viter xk ) + o( \lambda_k ) + \viter sk \bigr) \\
            &{}= f\bigl( \viter xk - \lambda_kD( \viter xk )\nabla f( \viter xk ) + o( \lambda_k ) \bigr) \\
            &{}= f( \viter xk ) - \lambda_k\nabla f( \viter xk )^TD( \viter xk )\nabla f( \viter xk ) + o( \lambda_k ) \\
            &{}= f( \viter xk ) - \lambda_k\| D( \viter xk )^{1 / 2}\nabla f( \viter xk ) \|^2 + o( \lambda_k ) \\
            &{}\leq f( \viter xk ) - \lambda_k\epsilon + o( \lambda_k ),
         \end{split}
      \end{equation*}
      where the first equality is obtained replacing $\viter x{k + 1}$ according to the definition of algorithm in~\eqref{eq:algo_2step} and assumption~\eqref{eq:ass_step1}, the second equality comes from $\| \viter sk \|/\lambda_k \to 0$ and the third from the smoothness and boundedness assumptions which allow us to use the Taylor expansion of $f$ around $f( \viter xk )$ in order to estimate $f( \viter x{k + 1} )$. The fourth equation is merely formal and the last equation finally comes from $\| D( \viter x{k} )^{1 / 2}\nabla f( \viter x{k} ) \|^2 \geq \epsilon > 0$.

      Now, let $\kappa \in \mathbb N$ be such that $k > \kappa \Rightarrow o( \lambda_k ) / \lambda_k < \epsilon / 2$. Then, iterating the above inequality from $\viter x\kappa$ we get:
      \begin{equation*}
         \begin{split}
            f( \viter x{\kappa + n} ) & {}\leq f( \viter x\kappa ) - \sum_{k = 0}^{n - 1}\lambda_{\kappa + k}\left\{\epsilon - \frac{o( \lambda_{\kappa + k} )}{\lambda_{\kappa + k}}\right\}\\
            & {}\leq f( \viter x\kappa ) - \frac\epsilon2\sum_{k = 0}^{n - 1}\lambda_{\kappa + k}.
         \end{split}
      \end{equation*}
      From this inequality we see that $\sum_{k = 0}^\infty \lambda_k = \infty$ implies that $f( \viter xk ) \to -\infty$.
   \end{proof}

   Now we use the convexity of the objective function and an extra a priori assumption in order to obtain global convergence of the algorithm. We shall not give the proofs here, instead referring, respectively, to~\cite[Propositions~7~and~5]{hed05} for demonstrations.

   \begin{prop}
      Suppose that beyond the assumptions of Proposition~\ref{prop:conv_weak}, we further have that $f( \viter xk )$ converges and that $f$ is strictly convex. Then
      \begin{equation*}
         f( \viter xk ) \to f^*,
      \end{equation*}
      where $f^*$ is the optimal value of $f$ over $\mathbb R_+^n$.
   \end{prop}

   \begin{prop}
      Suppose that beyond the assumptions of Proposition~\ref{prop:conv_weak}, we further have to assume that $\sum_{k = 0}^\infty \| o( \lambda_k ) \| < \infty$ and that $f$ is strictly convex. Then
      \begin{equation*}
         f( \viter xk ) \to f^*,
      \end{equation*}
      where $f^*$ is the optimal value of $f$ over $\mathbb R_+^n$.
   \end{prop}

   While these results may be useful in many cases, our interest is in exploring the circumstances where there may be several optima and we wish to select one among these, according to some criterion. Therefore, the strict convexity hypothesis is too restrictive (albeit superiorized techniques can be useful in the unique solution case too). Furthermore, summability of the perturbations could make the influence of the superiorization small in the overall optimization process.

   A difficulty in proving convergence results for algorithms of the form~\eqref{eq:algo_2step} is that the step direction can vanish at a non-optimal point. This fact leads to the possibility that a sequence that could converge to non-optimality. This possibility can not be ruled out because there is no way to guarantee that the error term is small compared to the components of the iterates in case $\iter{x_j}k \to 0$. The technical way found in the literature is to use summability of the perturbation term to obtain convergence of the objective value sequence. Then, assuming strict convexity one can prove convergence by finding out that the existence of a non-optimal limit point would imply the existence of another different limit point and that both of these are optimizers of the original problem with added constraints of the form $x_j = 0$. The next subsection analyses the alternative \textsc{ssaem}, which avoids these difficulties.

   \subsection{SSAEM-SUP}

   We now analyze iterative techniques of the form:
   \begin{equation}\label{eq:algo_2step_mod}
      \begin{split}
         \viter x{k + 1 / 2} &= \viter xk - \lambda_k \tilde D( \viter xk, \viter gk )\viter gk\\
         \viter x{k + 1} &= \viter x{k + 1 / 2}  + \viter sk.
      \end{split}
   \end{equation}
   Where
   \begin{equation}\label{eq:algo_2step_mod_approx}
      \viter gk = \nabla f( \viter xk ) + \vect \epsilon_k
   \end{equation}
   with
   \begin{equation}\label{eq:algo_2step_mod_vanerr}
      \vect \epsilon_k \to \vect 0
   \end{equation}
   and where $\tilde D( \vect x, \vect g )$ is the diagonal matrix given componentwise as
   \begin{equation*}
   \begin{split}
      \tilde D( \vect x, \vect g )_{jj} & {}:= \begin{cases}
                     \frac{\tau}{p_j} & \text{if}\quad x_j \leq \tau\quad\text{and}\quad g_j \leq 0\\
                     \frac{x_j}{p_j} & \text{otherwise}
                  \end{cases}\\
               & {}= \begin{cases}
                     \frac{x_j}{p_j} & \text{if}\quad g_j > 0\\
                     \frac{\max\{x_j, \tau \}}{p_j} & \text{otherwise.}
                 \end{cases}
   \end{split}
   \end{equation*}

   We have already discussed how to implement this kind of algorithm in the previous section, and it consists of a string-averaged scaled incremental gradient iteration where the scaling matrix does not contain too small entries, followed by a correction in those components which are small and have positive corresponding component $\iter{g_j}k$. The scaling matrix $\tilde D( \viter xk, \viter gk )$ can be seen as an approximation to $\hat D( \viter xk )$ where
   \begin{equation*}
      \hat D( \vect x )_{jj} := \begin{cases}
                     \frac{x_j}{p_j} & \text{if}\quad \displaystyle\frac{\partial f}{\partial x_j}( \vect x ) > 0\\
                     \frac{\max\{x_j, \tau\}}{p_j} & \text{otherwise.}
                 \end{cases}
   \end{equation*}
   Notice that, if $\vect x^* \in \mathbb R_+^n$, $\hat D( \vect x^* )^{1 / 2}\nabla f( \vect x^* ) = \vect 0$ is a necessary and sufficient optimality condition for $\vect x^*$ to be an optimizer of problem~\eqref{eq:prob_basic} given that $f$ is convex and $\tau > 0$. As a basic tool for further analysis, we give the practical approximation $\tilde D$ to $\hat D$ a more precise and useful mathematical characterization. We will simplify the notation by using $D^k := \tilde D( \viter xk, \viter gk )$.

   \begin{prop}\label{prop:mainFact}
      Suppose $f$ is convex and continuously differentiable, that $\{ \viter xk \} \subset \mathbb R_+^n$ generated by algorithm~\eqref{eq:algo_2step_mod} is bounded and that \eqref{eq:algo_2step_mod_approx} and \eqref{eq:algo_2step_mod_vanerr} hold. Then, for every $\epsilon > 0$, there are $\kappa \in \mathbb N$ and $\delta_\epsilon > 0$ such that for every $k \geq \kappa$ we have
      \begin{equation*}
         f( \viter xk ) \geq f^* + \epsilon \Rightarrow \| (D^k)^{1 / 2}\viter gk \|^2 \geq \delta_{\epsilon}.
      \end{equation*}
   \end{prop}
   \begin{proof}
      If the claim were false, there would exist $\epsilon > 0$ and a subsequence $\{ \viter x{l_k} \}$ such that $f( \viter x{l_k} ) \geq f^* + \epsilon$ and $\| (D^{l_k})^{1 / 2}\viter g{l_k} \| \to 0$. Because of the boundedness assumptions, there is no loss of generality in assuming that we have the limits $\viter x{l_k} \to \overline{\vect x}$ and $D^{l_k} \to \overline D$. Notice also that we have $\viter g{l_k} \to \nabla f( \overline{\vect x} )$.

      Furthermore, we claim that $\hat D( \overline{ \vect x} )\nabla f( \overline{\vect x} ) = \vect 0$. If this is not true, it is because at least one of the two cases below hold for a given index $j$:
      \begin{enumerate}
         \item $\overline x_j > 0$ and $\displaystyle\frac{\partial f}{\partial x_j}( \overline{\vect x} ) > 0$;\label{case:negative}
         \item $\displaystyle\frac{\partial f}{\partial x_j}( \overline{\vect x} ) < 0$\label{case:positive}.
      \end{enumerate}
      If Case~\ref{case:negative} were true, for large enough $k$ we would have $\iter{x_j}{l_k} \geq 1/2\overline x_j$ and $\iter{g_j}{l_k} \geq 1/2 \frac{\partial f}{\partial x_j}( \overline{\vect x} )$. Therefore, there would hold $D^{l_k}_{jj}\frac{\partial f}{\partial x_j}( \viter x{l_k} ) \geq 1/4D( \overline{\vect x} )_{jj}\frac{\partial f}{\partial x_j}( \overline{\vect x} ) > 0$, which contradicts $\| (D^{l_k})^{1 / 2}\viter g{l_k} \| \to 0$. If, on the other hand, Case~\ref{case:positive} holds, then, for large enough $k$, we have $\iter{g_j}{l_k} \leq 1/2 \frac{\partial f}{\partial x_j}( \overline{\vect x} )$ and therefore $D^{l_k}_{jj}\frac{\partial f}{\partial x_j}( \viter x{l_k} ) \leq 1/2\tau/p_j\frac{\partial f}{\partial x_j}( \overline{\vect x} ) < 0$, again a contradiction.

      Now, since $\overline{\vect x}$ satisfies $f( \overline{\vect x} ) \geq f^* + \epsilon$ it cannot be an optimal point. However, $\hat D( \overline{\vect x} )\nabla f( \overline{\vect x} ) = \vect 0$ is a necessary and sufficient condition for $\overline{\vect x}$ to be optimal if $f$ is convex. This contradiction is derived from the assumption that the specified $\{ \viter x{l_k} \}$ exists, and therefore the claim is proven.
   \end{proof}

   With this result in hand we can prove convergence of the method. The key idea is to notice that if $f( \viter xk )$ differs from $f^*$ by more than some threshold, then we can prove that the algorithm reduces the value of the objective function at small enough stepsize regimes. This implies in convergence because it results that once we are inside a sublevel set we cannot ``get away from it'' by much. The details are in the theorem below.

   \begin{theo}
      Assume $f$ is convex and continuously differentiable, the sequence of stepsizes satisfy
      \begin{equation*}
         \sum_{k = 0}^\infty\lambda_k = \infty, \quad\lambda_k \to 0^+\quad\text{and}\quad\frac{\| \viter sk \|}{\lambda_k} \to 0,
      \end{equation*}
      and that $\{ \viter xk \}$, as generated by Algorithm~\eqref{eq:algo_2step_mod} satisfying ~\eqref{eq:algo_2step_mod_approx}~and~\eqref{eq:algo_2step_mod_vanerr} is bounded and nonnegative. Then
      \begin{equation*}
         f( \viter xk ) \to f^*.
      \end{equation*}
   \end{theo}
   \begin{proof}
   Let us first provide an estimate of the decrease in function value based on the optimality measure $(D^k)^{1 / 2}\viter gk$:
   \begin{equation}\label{eq:reduction}
      \begin{split}
         f( \viter x{k + 1} ) & {}= f\left( \viter xk - \lambda_k\left [ D^k\viter gk + \frac{ \viter sk }{\lambda_k} \right] \right) \\
         &{}= f( \viter xk ) - \lambda_k\nabla f( \viter xk )^T\left [ D^k\viter gk + \frac{ \viter sk }{\lambda_k} \right] + o( \lambda_k ) \\
         &{}= f( \viter xk ) - \lambda_k{\viter gk}^TD^k\viter gk + o( \lambda_k ) \\
         &{}= f( \viter xk ) - \lambda_k\| ( D^k )^{1 / 2}\viter gk \|^2 + o( \lambda_k ).
      \end{split}
   \end{equation}
   In the above sequence, the first equality is from~\eqref{eq:algo_2step_mod}, the second from differentiability of $f$ and boundedness of $\{ D^k\viter gk + \viter sk / \lambda_k \}$, which is a consequence of the boundedness of $\{ \viter xk \}$ and of $\viter sk / \lambda_k \to 0$. The third equality is based on~\eqref{eq:algo_2step_mod_approx}~and~\eqref{eq:algo_2step_mod_vanerr}, while the fourth equality is merely formal.

   Now let us fix any $\epsilon > 0$ and assume that $\kappa$ is large for Proposition~\ref{prop:mainFact} to hold for this $\epsilon$ and also that $k \geq \kappa$ implies that $o( \lambda_k  ) / \lambda_k \leq \delta_\epsilon / 2$, where $\delta_\epsilon$ is from Proposition~\ref{prop:mainFact}.

   We then split in two cases:
   \begin{enumerate}
      \item $f( \viter xk ) > f^* + \epsilon$;\label{case:large}
      \item $f( \viter xk ) \leq f^* + \epsilon$.\label{case:small}
   \end{enumerate}

   In Case~\ref{case:large}, Proposition~\ref{prop:mainFact}, equation~\eqref{eq:reduction} and $o( \lambda_k  ) / \lambda_k \leq \delta_\epsilon / 2$ lead to
   \begin{equation}\label{eq:redLarge}
      \begin{split}
         f( \viter x{k + 1} ) & {}= f( \viter xk ) - \lambda_k\| ( D^k )^{1 / 2}\viter gk \|^2 + o( \lambda_k )\\
         & {}\leq f( \viter xk ) - \lambda_k\delta_\epsilon + \lambda_k\frac{\delta_\epsilon}2\\
         & {}\leq f( \viter xk ) - \lambda_k\frac{\delta_\epsilon}2.
      \end{split}
   \end{equation}
   Therefore, because $\sum_{k = 0}^\infty \lambda_k = \infty$ we know that the either the second case occurs infinitely many times or $f( \viter xk ) \to -\infty$. If this second possibility is true, the claim is proven, if not, we proceed with the argument.

   Now if Case~\ref{case:small} holds, then we have
   \begin{equation*}
      f( \viter x{k + 1} ) = f\bigl( \viter xk + O( \lambda_k ) \bigr) \leq  f\bigl( \viter xk ) + \lambda_k M
   \end{equation*}
   for some large enough $M$. Therefore, a consequence of the fact that we reach this case infinitely many times coupled with \eqref{eq:redLarge} for Case~\ref{case:negative} is that after reaching the sublevel set $\lev_{\epsilon}( f ) := \{\vect x : f( \vect x ) \leq f^* +\epsilon \}$, say in iterate $k_0$, the algorithm never escapes from the larger, but ever reducing, sublevel set $\lev_{\epsilon + \lambda_{k_1}}( f )$ for some $k_1 \geq k_0$. Thus
   \begin{equation*}
      \limsup_{k \to \infty} f( \viter xk ) \leq f^* + \epsilon.
   \end{equation*}
   But since $\epsilon > 0$ was arbitrary and $f( \viter xk ) \geq f^*$, we then have the claim proven.
   \end{proof}

%
%
%

   \section{Numerical Experimentation}

   The numerical experiments are divided in two independent sets. In the first of these, the \textsc{saem}~\cite{hcc14} is compared with the \textsc{em}~\cite{vsk85} considering two kind of superiorization sequences for non-negatively constrained Total Variation as a secondary criteria. The first superiorization sequence is produced according to Algorithm \ref{super-seq}, following the ideas presented in \cite{Garduno2014}, the second superiorization sequence is produced by the \textsc{fgp} algorithm from \cite{bet09b}. 

   Moreover, this first set of experiments considers 15 repetitions of the error simulating procedure where the \textsc{mse} and \textsc{ssim}~\cite{wbs04} figures of merit of the reconstructed image of each algorithmic variation are computed alongside with other numerical and performance indicators. These experiments were done on a quad core Intel i5-4570S \textsc{cpu} @2.9GHz, with 32GB DDR3 memory.

   The second test is used in order to assess the viability of the \textsc{ssaem} with or without superiorization for large tomographic reconstruction from actual data. It uses a specialized \textsc{gpu} implementation which is efficient only for $s = 1$. The objective of this set of experiments is to both give a proof of concept showing that the algorithm is efficient in practical applications and to assess the effects of superiorization in its performance for use with the maximum likelihood model for transmission tomography. In this case a GeForce GTX 745 was used.

   \subsection{\textsc{saem}}

   This technique was used, as already mentioned, in order to solve problems of the form
   \begin{equation*}
      \begin{split}
         \min &\quad \sum_{i = 1}^m\left\{ (R\vect x)_i - b_i \log( R\vect x )_i \right\}\\
         \st &\quad \vect x \in \mathbb R_+^n,
      \end{split}
   \end{equation*}
   which can be interpreted both as the maximum likelihood model for emission tomography and as the minimal Kullback-Leibler distance model for a general non-consistent non-negative system of equations. In this case, each $f_i$ was only one term of the above sum, i.e.,
   \begin{equation*}
      f_i( \vect x ) = (R\vect x)_i - b_i \log( R\vect x )_i
   \end{equation*}
   and, therefore,
   \begin{equation*}
      \nabla f_i( \vect x ) = R_i^T\left( 1 - \frac{b_i}{( R\vect x )_i} \right),
   \end{equation*}
   where $R_i$ is the $i$-th line of $R$.

   \paragraph{Tomographic setup} For this first set of tests we utilized 32 angles with 182 line integrals each. The reconstructed images and the original numeric phantoms have dimensions of $128 \times 128$ pixels. Poisson noise was added to this sparse-angle simulated acquisition resulting in a signal to noise ratio of around $18$dB in the data for each of the 15 repetitions.

   \subsubsection{Algorithmic parameters}

   \paragraph{Starting image} All of the algorithms, including \textsc{saem} with or without superiorization, were started from a uniform image $\iter x0_j = \phi$ such that $\sum_{i = 1}^m( R\viter x0 )_i = \sum_{i = 1}^m b_i$. It can be noticed that the appropriate value for this to hold is
   \begin{equation*}
      \phi = \frac{\sum_{i = 1}^m b_i}{\sum_{i = 1}^m( R\vect 1 )_i},
   \end{equation*}
   where $\vect 1$ is a vector in $\mathbb R^n$ with all entries equal to $1$.

   \paragraph{Strings formation and weights} We denote as \textsc{saem}-$s$ the \textsc{saem} algorithm with $s$ strings, and as \textsc{saem}-$s$-\textsc{tvs} and \textsc{saem}-$s$-\textsc{tvs}-\textsc{fgp} its standard superiorized version and its \textsc{fgp} superiorized version, respectively. All of \textsc{saem}-$s$ versions, with or without superiorization, used weights $w_i = 1 / s$, thereby maintaining the unweighted likelihood model. Strings were selected by randomly shuffling the data and partitioning into $s$ strings of equal length (depending on divisibility of the data size by the number of strings, some strings actually contained one element more than others). All algorithms were ran for several iterations, however the final result was assumed when $f( \viter xk ) \leq 400$. This value was experimentally chosen corresponding to a point where all algorithms are close to their minimum MSE (see Figures \ref{EMvsSAEM} and \ref{EMvsSAEM-TVS}).

   \paragraph{Scaling matrix} The coefficients $p_i$ of $D( \vect x )$ must be selected, we have followed~\cite{hcc14}:
   \begin{equation*}
      p_j = \sum_{i = 1}^n a_{ij} \quad j \in \{ 1, 2, \dots, n \}.
   \end{equation*}
   This scaling was selected because then \textsc{saem}-$m$ with fixed stepsize $\lambda = m$ corresponds to the \textsc{em} algorithm.

   \paragraph{Stepsize sequence} The stepsize sequence $\lambda_k$ was chosen following~\cite{hcc14}: for \textsc{saem}-$s$ (i.e., $s$ strings) and its superiorized versions we have used
   \begin{equation*}
      \lambda_k = \frac{\lambda_{0, s}}{k^{0.51} / s + 1},
   \end{equation*}
   where $\lambda_{0, s}$ was the largest value such that \textsc{saem}-$s$ would lead to a positive $\viter x1$.

   \paragraph{Superiorization sequences} For the standard superiorization sequence, the step $\beta_{n,k}$ is unknown a priory and depends on the choice of the nonascending direction $\viter v{n,k}$. So the size of the vector  $\viter s{k}$ depends on $\beta_0$ and $N$. The larger these two variables the larger is the visual effect in the superiorization sequence. We have used $\beta_0 = 1$ and $N = 10$ or $N = 20$ for \textsc{em} and \textsc{saem}, respectively, and we have fixed $\alpha = 0.95$.

   For the superiorization sequence produced by the proximal operator via FGP, the size of $\viter s{k}$ depends on the $\gamma_k$ used in \eqref{eq:TV-PROX}.  The $\gamma_k$ chosen for this experiments is:
   \begin{equation} \label{eq:gamma-fgp}
      \gamma_k = \dfrac{\gamma_0}{(k+1)^{1+eps}}
   \end{equation}
   where $eps$ is the machine epsilon\footnote{The machine epsilon is an upper bound on the relative error due to rounding in floating point arithmetic. For double precision it is around 2.22e-16.}. The values used for $\gamma_0$ used were manually adjusted to $0.15$ for \textsc{em}-\textsc{tvs}-\textsc{fgp} and $0.3$ for \textsc{saem}-\textsc{tvs}-\textsc{fgp}, which were values that produced a competitive performance. The sequence with step from \eqref{eq:gamma-fgp} is also summable, which is a requisite from the convergence analysis. We expect the sequence produced by the proximal operator to be more efficient than subgradient, in the sense that the superiorized sequence is closer to the Pareto optimal curve, but without much more computational cost per iteration.

   \subsubsection{Numerical results}

   The tests compared \textsc{em} and \textsc{saem}, considering the use or not of a superiorization sequence. In Figure~\ref{EMvsSAEM} we show the comparison with the standard superiorization sequence. The plot in Figure~\ref{EMvsSAEM-ERR} shows one realization of the experiment where the estimation error from the current solution to the original phantom image $\vect x^*$. The green boxes in the error curves correspond to the first image in the iteration process, of each method, to satisfy the stopping criteria. The reconstructed images from one of the 15 repetitions are shown in Figure~\ref{EMvsSAEM-images} and Table~\ref{tab:imag_analysis_saem} brings some statistical information about experiment.

   Figure~\ref{EMvsSAEM-TVKL} shows a curve of Kullback Leibler distance versus Total Variation measurement at each iteration. This figure, which resembles an L-curve~\cite{han92}, helps to see how much the Total Variation measurement increases while the iterations runs reducing the Kullback Leibler distance, basically moving from the lower right to the upper left of Figure \ref{EMvsSAEM-TVKL}. The points corresponding to the images satisfying the stopping criteria are also marked in green boxes on the plots.

   \begin{figure}%
      \centering%
      \subfigure[Estimation Error $\| \viter xk - \vect x^* \|$]{\label{EMvsSAEM-ERR}%
         \includegraphics[width=0.5\textwidth,height=0.3\textwidth]{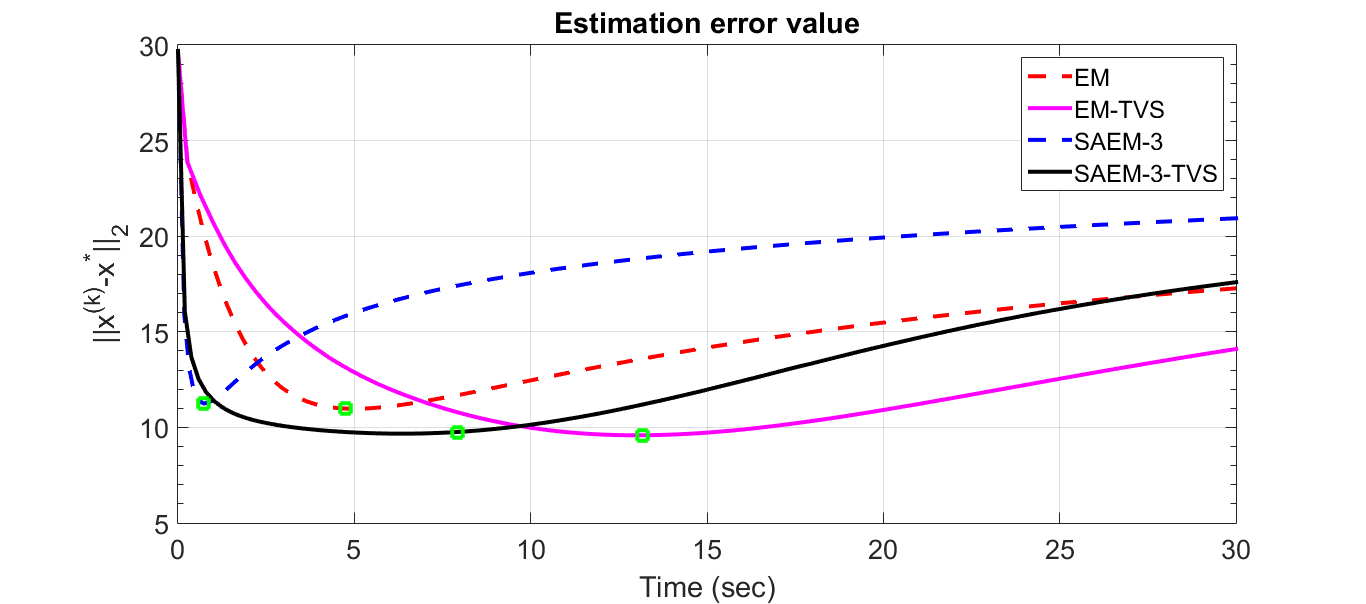}%
      }%
      \subfigure[Total Variation vs Kullback-Leibler Curve]{\label{EMvsSAEM-TVKL}%
         \includegraphics[width=0.5\textwidth,height=0.3\textwidth]{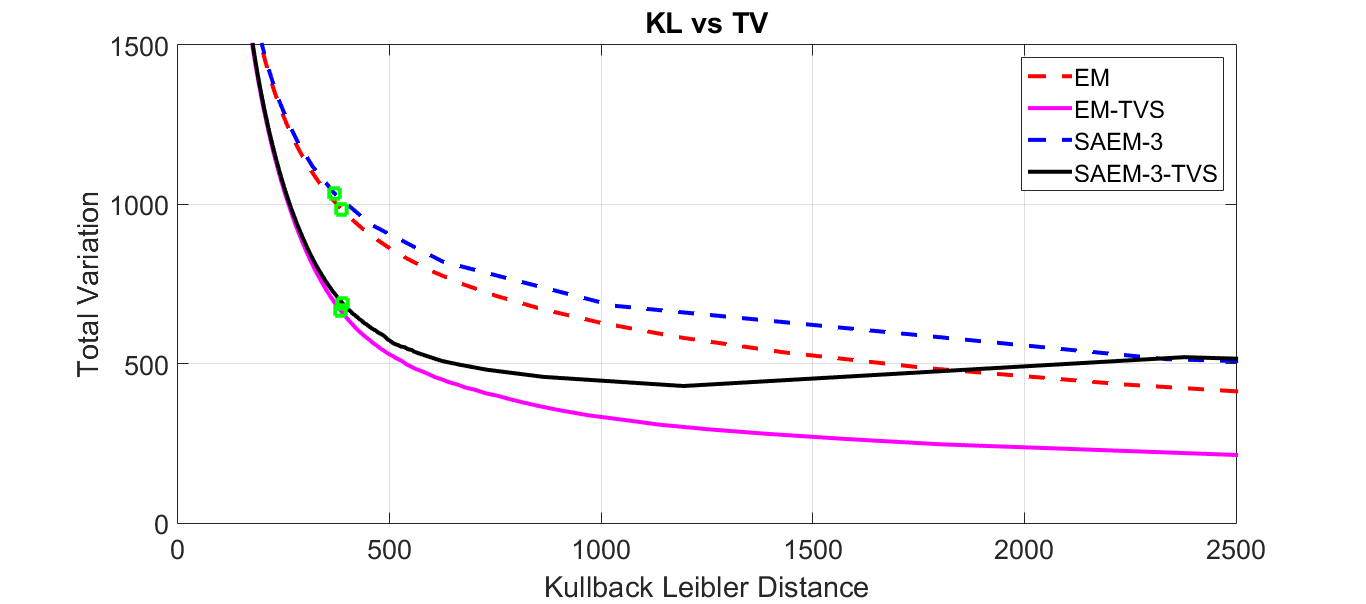}%
      }%
      \caption{Comparison between EM and SAEM with and without superiorization.}\label{EMvsSAEM}%
   \end{figure}%
   \begin{figure}%
      \centering%
      \subfigure[Estimation Error $|| \viter xk - \vect x^* \|$]{\label{EMvsSAEM-TVS-ERR}%
         \includegraphics[width=0.5\textwidth,height=0.3\textwidth]{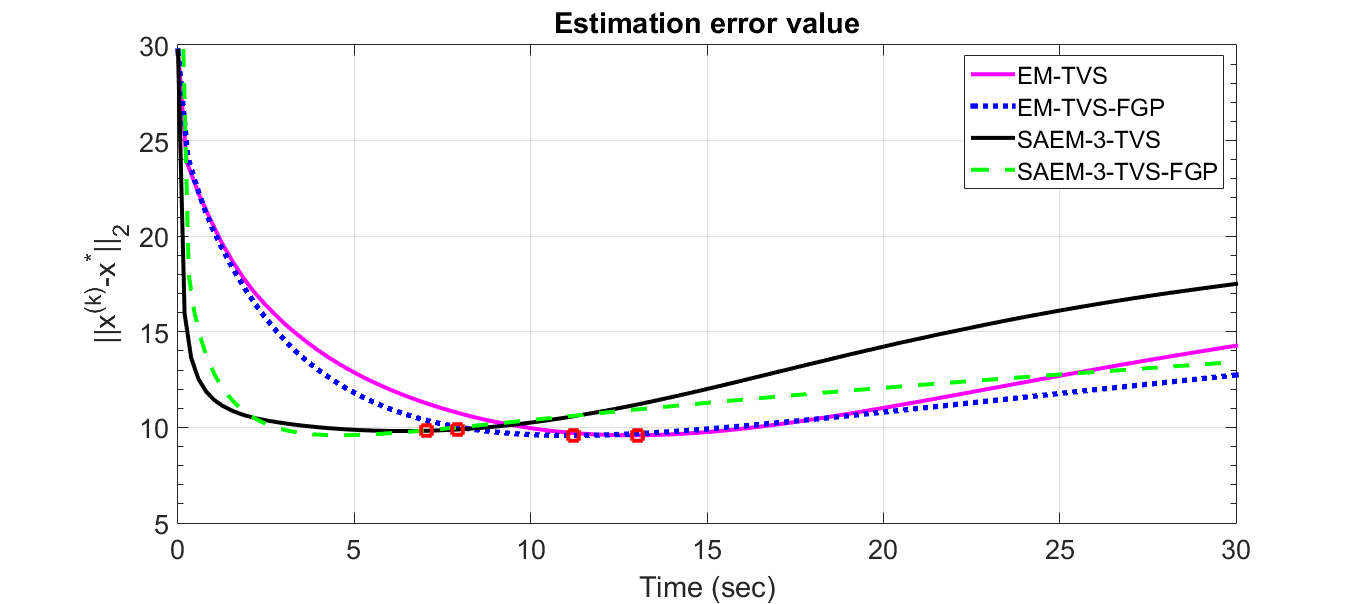}%
      }%
      \subfigure[Total Variation vs Kullback-Leibler Curve ]{\label{EMvsSAEM-TVS-TVKL}%
         \includegraphics[width=0.5\textwidth,height=0.3\textwidth]{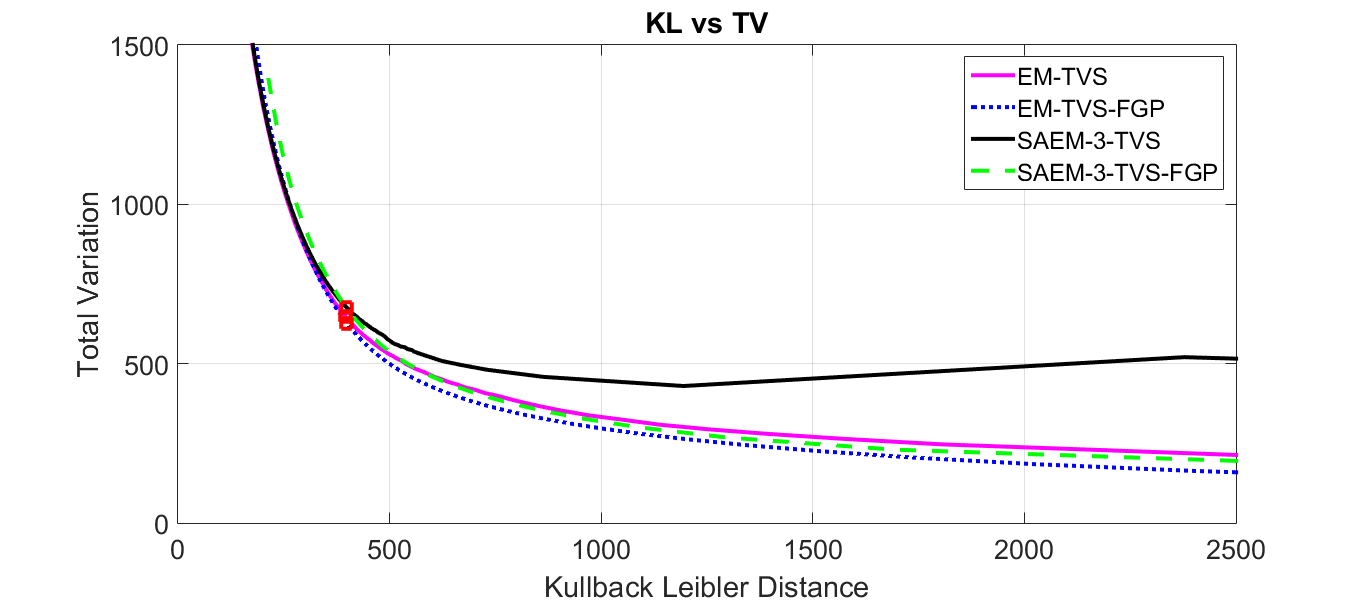}%
      }%
      \caption{Comparison between EM and SAEM with standard superiorization and FGP superiorization.}\label{EMvsSAEM-TVS}%
   \end{figure}

  In Figure~\ref{EMvsSAEM-TVS} we compare the effect of the two superiorization sequences, the standard one against the proximal, produced by FGP. The plot in Figure~\ref{EMvsSAEM-TVS-ERR} shows the estimation error from the current solution to the original phantom image $\vect x^*$. Now the points corresponding to images satisfying the stopping criteria are marked as red boxes. The plot in Figure~\ref{EMvsSAEM-TVS-TVKL}  shows a curve of Kullback Leibler distance and Total Variation measurement. This plot also shows, as iteration runs, an evolution of the results moving the lower right to the upper left of Figure~\ref{EMvsSAEM-TVS-TVKL}. Note that the final result according to the stopping criteria never reaches the upper left part, it ends in the red boxes which is the first iteration where $f( \viter xk ) \leq 400 $ .

   \begin{figure}[t]%
      \newcommand{\wdt}{0.4\textwidth}%
      \subfigure[EM]{\label{EM}%
         \includegraphics[width=\wdt]{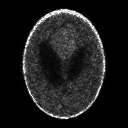}%
      }\hfill%
      \subfigure[SAEM-3]{\label{SAEM}%
        \includegraphics[width=\wdt]{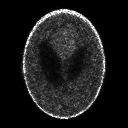}%
      }\\%
      \subfigure[EM-TVS]{\label{EM-TVS}%
         \includegraphics[width=\wdt]{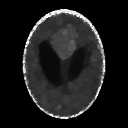}%
      }\hfill%
      \subfigure[SAEM-3-TVS]{\label{SAEM-TVS}%
         \includegraphics[width=\wdt]{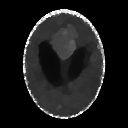}%
      }\\
      \subfigure[EM-TVS-FGP]{\label{EM-TVS-FGP}%
         \includegraphics[width=\wdt]{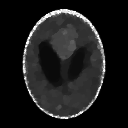}%
      }\hfill%
      \subfigure[SAEM-3-TVS-FGP]{\label{SAEM-TVS-FGP}%
         \includegraphics[width=\wdt]{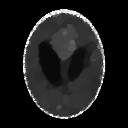}%
      }%
      \caption{Visual Comparison between EM and SAEM images.}\label{EMvsSAEM-images}%
   \end{figure}%

   \begin{table}%
      \centering \resizebox{\textwidth}{!}{\small%
      \begin{tabular}{rcccccc}%
                          & \textsc{em} & \textsc{saem}-3 & \textsc{em}-\textsc{tvs} & \textsc{saem}-3-\textsc{tvs}%
                          & \textsc{em}-\textsc{tvs}-\textsc{fgp} & \textsc{saem}-3-\textsc{tvs}-\textsc{fgp} \\%
         \hline%
         \textsc{kl}      & $389.5 \pm 21.9$  & $329.5 \pm 103.4$  & $395.9 \pm 5.7$ 
          				  & $398.6 \pm 3.0$   & $397.8 \pm 4.4$    & $398.1 \pm 3.7$\\%
         \textsc{tv}      & $935.2 \pm 62.0$  & $1076.9 \pm 165.4$ & $612.7 \pm 50.6$
                          & $670.9 \pm 54.1$  & $592.1 \pm 49.7$   & $653.0 \pm 54.1$\\%
         \textsc{mse}     & $10.6 \pm 0.5$    & $11.0 \pm 0.6$     & $9.2 \pm 0.6$   
                          & $9.4 \pm  0.7$    & $9.2 \pm 0.6$      & $9.4 \pm 0.8$\\%
         \textsc{ssim}    & $0.72 \pm 0.01$   & $0.71 \pm 0.01$    & $0.85 \pm 0.01$   
                          & $0.85 \pm 0.01$   & $0.85 \pm 0.01$    & $0.86 \pm 0.01$\\%
         iter.            & $21.2 \pm 1.3$    & $4.8 \pm 1.2$      & $47.1 \pm 4.5$           
         			      & $24.5 \pm 9.0$    & $46.0 \pm 5.8$     & $30.2 \pm 6.3$\\%
         time             & $2.9 \pm 0.2$s    & $0.4 \pm  0.2$s    & $8.8 \pm 0.9$s           
                  		  & $4.6 \pm 1.7$s    & $7.1 \pm 0.9$s     & $3.8 \pm 0.9$s\\%
         \hline%
      \end{tabular}%
      }%
      \caption{Statistical indicators at the point where the stopping criterion was reached (only for \textsc{ssim} higher is better). The algorithms were repeated 15 times with different noise. The values shown are the sample means with  99\% confidence interval.}\label{tab:imag_analysis_saem}%
   \end{table}

   \subsection{\textsc{ssaem}}

   Now, the problem is to reconstruct images from transmission data. Therefore, the maximum likelihood model is
   \begin{equation*}
      \begin{split}
         \min & \quad \sum_{i = 1}^m\left\{ \beta_ie^{-( R\vect x )_i} - \alpha_i \log\left( e^{-(R \vect x )_i} + \rho_i\right) \right\}\\
         \st  & \quad \vect x \in\mathbb R_+^n.
      \end{split}
   \end{equation*}
   In this case, we split the data in $s$ subsets (following the same construction rules than the strings) $S_l$, $l \in \{ 1, 2, \dots, s \}$, such that
   \begin{equation*}
      f( \vect x ) = \sum_{l = 1}^sF_l( \vect x ),
   \end{equation*}
   where
   \begin{equation*}
      F_l( \vect x ) = \sum_{i \in S_l}\left\{ \beta_ie^{-( R\vect x )_i} - \alpha_i \log\left( e^{-(R \vect x )_i} + \rho_i\right) \right\}.
   \end{equation*}
   Therefore
   \begin{equation*}
      \frac{\partial F_l}{\partial x_j}( \vect x ) = \sum_{i \in S_l}r_{ij}\beta_ie^{-(R\vect x)_i}\left\{ \frac{\alpha_i}{\beta_ie^{-(R\vect x)_i} + \rho_i} - 1 \right\}.
   \end{equation*}

   \paragraph{Tomographic setup} Data for this set of experiments was obtained at the Imaging Beamline (\textsc{imx}) of the Brazilian National Synchrotron Light Source (\textsc{lnls}) through illumination of an apple seed by \textsc{x}-rays. Each radiographic image was $2048 \times 2048$ pixels in size and the sample was rotated by $\pi / 512$ between each of the $512$ images, thereby totaling, for each slice to be reconstructed, $512$ views times $2048$ rays of tomographic count data $\alpha_i$. Reconstructed images have dimensions of $2048 \times 2048$ pixels.

   Blank scans were measured before and after the actual tomographic measurements of the sample and then data $\beta_i$ was obtained by linear interpolation between these two measurements. This is necessary because because \textsc{lnls} is a 2${}^\text{nd}$ generation synchrotron that does not operate in a top-up mode, i.e., the storage ring is not injected periodically in order to maintain current constant. Therefore, flux is not constant and decrease exponentially with time. As this measurement was obtained with a short exposure time, we expect that ring current (hence flux) behaves approximately as a linear model. A 4${}^\text{th}$ generation synchrotron source, such as the forthcoming brazilian source, Sirius, will operate in a top-up mode with more flux, i.e., samples will be measured with a short exposure time.

   It is important to note that iterative methods such as those presented in this paper are extremely attractive for ultra-fast tomographic experiments, with an ultra-short exposure time and a considerable amount of noise. This is the case of soft-tissue samples where low dose should be considered and a analytical reconstruction method like filtered backprojection \cite{nat86} produce low quality reconstructed images with strong streak artifacts.

   \subsubsection{Parameter Selection}

   \paragraph{Scaling matrix} For this algorithm we have used
   \begin{equation*}
      p_j = \sum_{i = 1}^m r_{ij}\{ \alpha_i - \rho_i \},
   \end{equation*}
   which provides a reasonable scaling in the sense that if the approximation $\alpha_i \approx \beta_ie^{-(R\viter xk)_i} + \rho_i$ holds, then iterations of the form $\viter x{k + 1} = \viter xk - D( \viter xk )\nabla f( \viter xk )$ become close to a multiplicative method similar to the \textsc{em} for emission tomography, and thus it can be written approximately as a product of coordinates of $\viter xk$ with corresponding ratios of expected by observed quantities. The lower threshold was set to be $\tau = 10^{-14}$.

   \paragraph{Subsets formation} Notation for \textsc{ssaem} has a different meaning, where \textsc{ssaem}-$s$ represents \textsc{ssaem} using $s$ subsets of data and always only one string. The construction of the subsets was done by including all rays in a sequence of views. This sequence of views had approximately the same size on each subset. Subset processing within the string was done selecting a different random permutation for each iteration. Because only one string was used, the experimented algorithm could be described as a less general stabilized \textsc{bramla}. This is because it was suitable for massive paralelization in a \textsc{gpu} which was neccessary given the large size of the dataset. However, if more than one \textsc{gpu} were to be used, one string for each of them could be a reasonable way of taking advantage of the multi-\textsc{gpu} setup.

   \paragraph{Stepsize sequence} As stepsize sequence for \textsc{ssaem}-$s$ and its superiorized version we have used
   \begin{equation*}
      \lambda_k = \frac{\lambda_{0, s}}{( ks + 1 )^{0.25}},
   \end{equation*}
   where $\lambda_{0, s}$ was the largest value such that the first iteration of \textsc{ssaem}-$s$ would not lead to a negative value.

   \paragraph{Superiorization sequence} The method termed \textsc{ssaem}-\textsc{tv}-$s$ corresponds to the superiorized version of \textsc{ssaem}-$s$ where the superiorization was obtained by application of~\eqref{eq:proj_subgrad} to the result of each \textsc{ssaem}-$s$ iteration. In this case, $N = 50$ and the sequence $\{ \gamma_k \}$ was chosen to be
   \begin{equation*}
      \gamma_k = \frac{\gamma_{0, s}}{( ks + 1 )^{0.35}}.
   \end{equation*}
   Now, $\gamma_{0, s}$ was selected so that
   \begin{equation*}
      \frac{\| \viter x{1 / 2} - \viter x1 \|}{\| \viter x0 - \viter x{1 / 2} \|} \approx 10^{-2}
   \end{equation*}
   in the following manner. First, $\viter x{1/2}$ is computed normally using the stepsize $\lambda_0$ as described above. Then, a tentative value $\tilde\mu_{0, s} = 1$ is used in order to compute $\viter{\tilde x}1$. Then we set
   \begin{equation*}
      \mu_{0, s} = 10^{-2}\frac{\| \viter x0 - \viter x{1 / 2} \|}{\| \viter x{1 / 2} - \viter x1 \|}.
   \end{equation*}

   \subsubsection{Numerical results}
   \begin{figure}
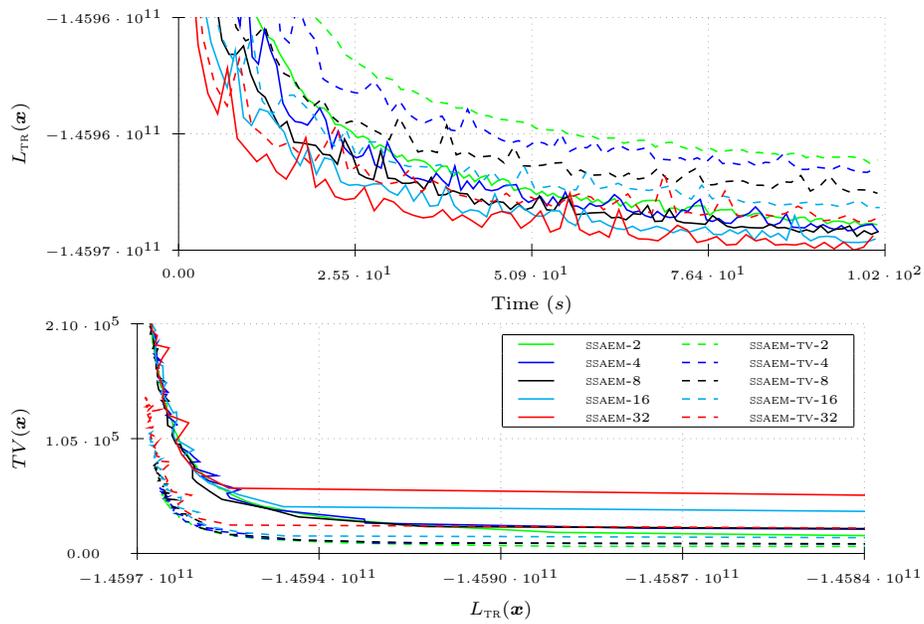

      \def\scale{1}%
      \def\lwidth{0.6pt}%
      \input{inc_times.tex}%
      \centering%
      \setlength{\grftotalwidth}{\textwidth}%
      \scalebox{\scale}{%
      \begin{grfgraphic*}[30]{%
         \grfxaxis{\xaxisValues}{\xaxisLabels}%
         \grfyaxis{\yaxisValues}{\yaxisLabels}%
         \grfxlabel{\scriptsize Time ($s$)}%
         \grfylabel[L]{\scriptsize$L_{\textsc{tr}}( \vect x )$}%
      }%
         {%
            \def\grfymax{1}%
            \grfxgrid[grfaxisstyle,gray,dotted]{\xaxisValues}{\xaxisValues}%
            \grfygrid[grfaxisstyle,gray,dotted]{\yaxisValues}{\yaxisValues}%
         }%
         \draw[green,line width=\lwidth] plot file {inc_times_2.data};%
         \draw[blue,line width=\lwidth] plot file {inc_times_4.data};%
         \draw[black,line width=\lwidth] plot file {inc_times_6.data};%
         \draw[cyan,line width=\lwidth] plot file {inc_times_8.data};%
         \draw[red,line width=\lwidth] plot file {inc_times_10.data};%
         \draw[green,dashed,line width=\lwidth] plot file {inc_times_3.data};%
         \draw[blue,dashed,line width=\lwidth] plot file {inc_times_5.data};%
         \draw[black,dashed,line width=\lwidth] plot file {inc_times_7.data};%
         \draw[cyan,dashed,line width=\lwidth] plot file {inc_times_9.data};%
         \draw[red,dashed,line width=\lwidth] plot file {inc_times_11.data};%
      \end{grfgraphic*}
      }\\%
      \input{inc_tv_obj.tex}%
      \scalebox{\scale}{%
      \begin{grfgraphic*}[0.07685]{%
         \grfxaxis{\xaxisValues}{\xaxisLabels}%
         \grfyaxis{\yaxisValues}{\yaxisLabels}%
         \grfxlabel{\scriptsize $L_{\text{\textsc{tr}}}( \vect x )$}%
         \grfylabel[L]{\scriptsize $TV( \vect x )$}%
      }%
         {%
            \def\grfymax{1}%
            \grfxgrid[grfaxisstyle,gray,dotted]{\xaxisValues}{\xaxisValues}%
            \grfygrid[grfaxisstyle,gray,dotted]{\yaxisValues}{\yaxisValues}%
         }%
         \draw[green,line width=\lwidth] plot file {inc_tv_obj_2.data};%
         \draw[blue,line width=\lwidth] plot file {inc_tv_obj_4.data};%
         \draw[black,line width=\lwidth] plot file {inc_tv_obj_6.data};%
         \draw[cyan,line width=\lwidth] plot file {inc_tv_obj_8.data};%
         \draw[red,line width=\lwidth] plot file {inc_tv_obj_10.data};%
         \draw[green,dashed,line width=\lwidth] plot file {inc_tv_obj_3.data};%
         \draw[blue,dashed,line width=\lwidth] plot file {inc_tv_obj_5.data};%
         \draw[black,dashed,line width=\lwidth] plot file {inc_tv_obj_7.data};%
         \draw[cyan,dashed,line width=\lwidth] plot file {inc_tv_obj_9.data};%
         \draw[red,dashed,line width=\lwidth] plot file {inc_tv_obj_11.data};%
         {%
            \tiny%
            \grftablegend[tr,fill=white]{|clcl|}{%
                  \hline%
                  \grflegsymbol[2em]{\grflegline[green,line width=\lwidth]} & \textsc{ssaem}-2 & \grflegsymbol[2em]{\grflegline[green,dashed,line width=\lwidth]} & \textsc{ssaem-tv}-2\\
                  \grflegsymbol[2em]{\grflegline[blue,line width=\lwidth]} & \textsc{ssaem}-4 & \grflegsymbol[2em]{\grflegline[blue,dashed,line width=\lwidth]} & \textsc{ssaem-tv}-4\\
                  \grflegsymbol[2em]{\grflegline[black,line width=\lwidth]} & \textsc{ssaem}-8 & \grflegsymbol[2em]{\grflegline[black,dashed,line width=\lwidth]} & \textsc{ssaem-tv}-8\\
                  \grflegsymbol[2em]{\grflegline[cyan,line width=\lwidth]} & \textsc{ssaem}-16 & \grflegsymbol[2em]{\grflegline[cyan,dashed,line width=\lwidth]} & \textsc{ssaem-tv}-16\\
                  \grflegsymbol[2em]{\grflegline[red,line width=\lwidth]} & \textsc{ssaem}-32 & \grflegsymbol[2em]{\grflegline[red,dashed,line width=\lwidth]} & \textsc{ssaem-tv}-32\\
                  \hline%
            }%
         }%
      \end{grfgraphic*}%
      }%
      \caption{Comparison of several \textsc{ssaem} variations and its respective superiorized versions. Top: Objective function value evolution over time. Bottom: objective function versus total variation. Legend is the same for both graphics.}\label{fig:ssaem_kl_time}%
   \end{figure}
   \begin{figure}
      \input{seed_inset_data.tex}%
      \centering%
      \setlength{\grftotalwidth}{\textwidth}%
      \begin{grfgraphic}{%
         \def\grfxmin{-1.125}\def\grfxmax{1}%
         \def\grfymin{-1.5}\def\grfymax{0.615}%
      }%
         \node[anchor=south,rotate=90,inner sep=\grflabelsep] at (-1.0,0.0) {\scriptsize$2048\times2048$};%
         \node[anchor=south,rotate=90,inner sep=\grflabelsep] at (-1.0,-1.0) {\scriptsize$256\times256$};%
         \node[anchor=south,inner sep=\grflabelsep] at (-0.5,0.5) {\scriptsize\textsc{ssaem-32}};%
         \node[anchor=south,inner sep=\grflabelsep] at (0.5,0.5) {\scriptsize\textsc{ssaem-tv-32}};%
         \node[inner sep=0pt] at (-0.5,0.0) {\includegraphics[width=\grfxunit]{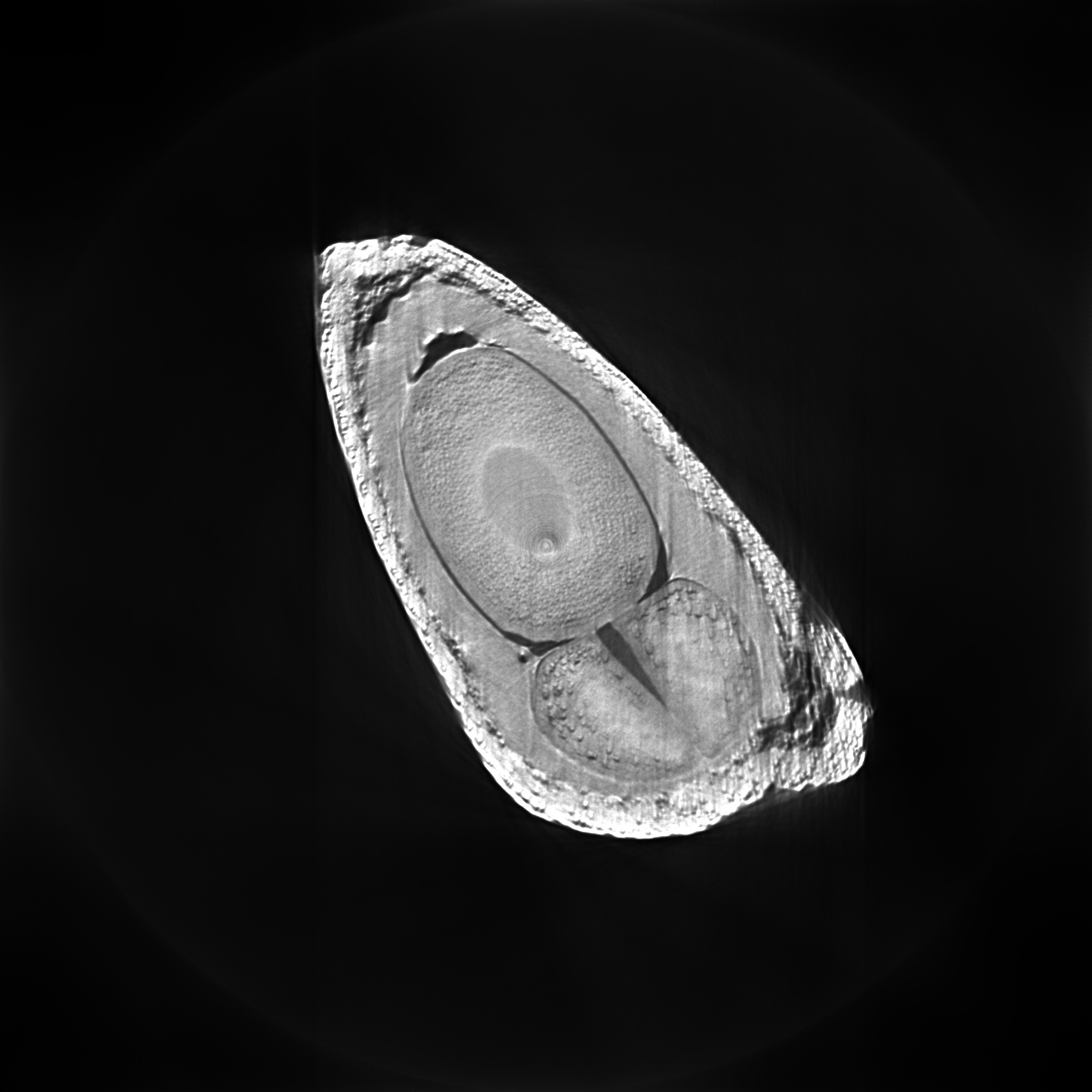}};%
         \node[inner sep=0pt] at ( 0.5,0.0) {\includegraphics[width=\grfxunit]{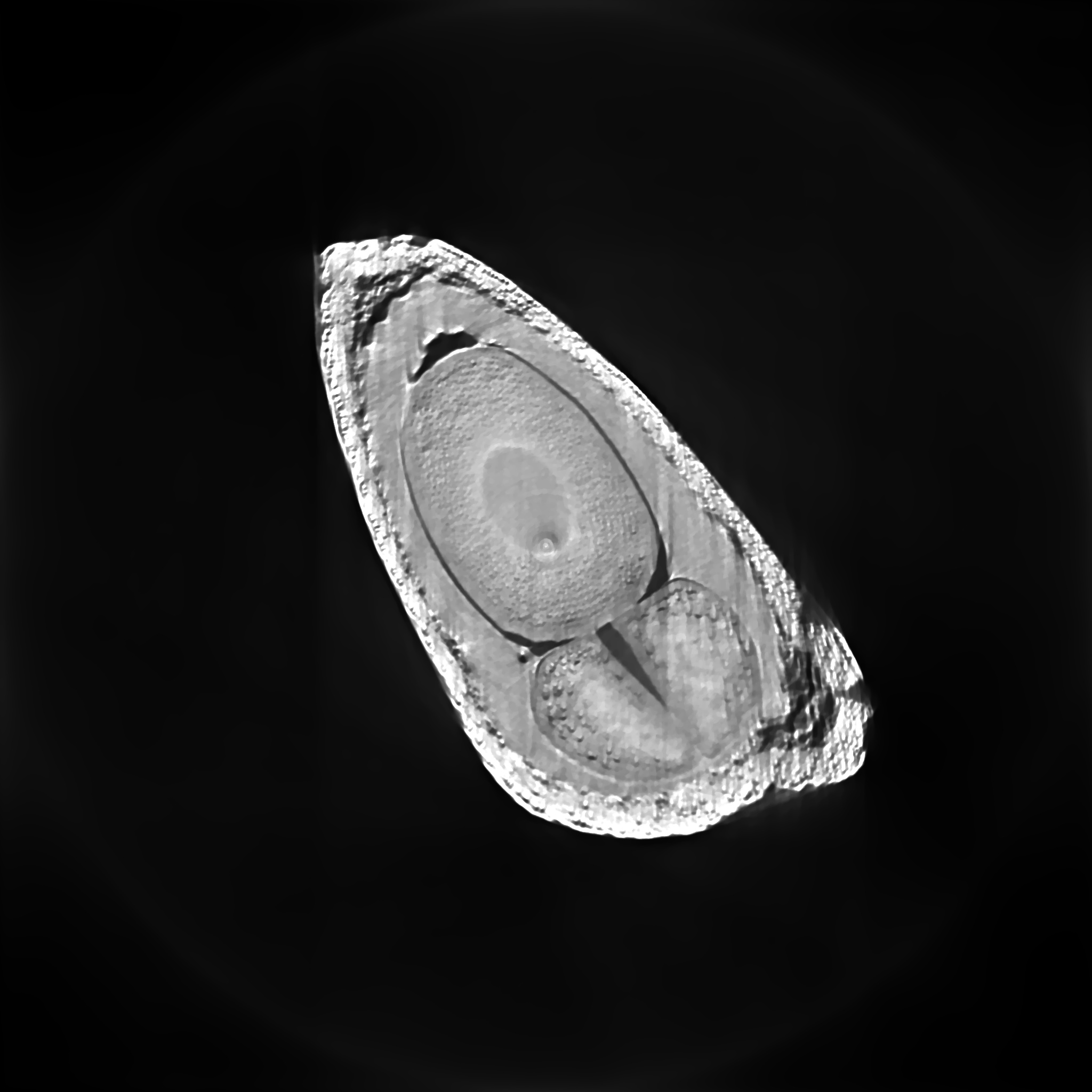}};%
         \draw[grfaxisstyle,gray] (-1.0,-0.5) rectangle (0.0,0.5);%
         \draw[grfaxisstyle,gray] (1.0,-0.5) rectangle (0.0,0.5);%
         \node[inner sep=0pt] at (-0.5,-1.0) {\includegraphics[width=\grfxunit]{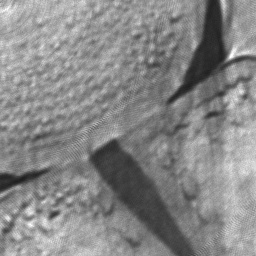}};%
         \node[inner sep=0pt] at ( 0.5,-1.0) {\includegraphics[width=\grfxunit]{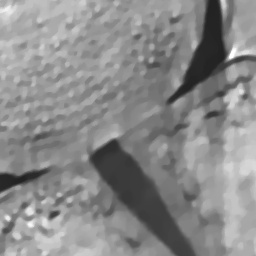}};%
         \begin{scope}[xshift=-1\grfxunit,yshift=-1.5\grfyunit]
            \input{contours_sub.tex}%
         \end{scope}
         \begin{scope}[xshift=0\grfxunit,yshift=-1.5\grfyunit]
            \input{contours_sub_tv.tex}%
         \end{scope}
         \draw[grfaxisstyle,gray] (-1.0,-1.5) rectangle (0.0,-0.5);%
         \draw[grfaxisstyle,gray] (1.0,-1.5) rectangle (0.0,-0.5);%
         \draw[grfaxisstyle,white] (\eggInsetTLX,\eggInsetTLY) rectangle (\eggInsetBRX,\eggInsetBRY);
         \draw[grfaxisstyle,white,shift={(-1,0)}] (\eggInsetTLX,\eggInsetTLY) rectangle (\eggInsetBRX,\eggInsetBRY);
      \end{grfgraphic}%
      \caption{Reconstructions of apple seed slice from synchrotron radiation transmission data. Top-row: full slice images. Bottom row: details, with overlaid isocontour and location in respective images above shown as a white square. These images are the first iterates to satisfy $L_{\textsc{tr}}( \viter xk ) \leq -1.4596 \cdot 10^{11}$.}\label{fig:seedImages}%
   \end{figure}

   The left side of Figure~\ref{fig:ssaem_kl_time} shows the objective function value as a function of computation time for \textsc{ssaem}-$s$ and \textsc{ssaem}-\textsc{tv}-$s$ with $s \in \{ 2, 4, 8, 16, 32 \}$. Notice that increasing the number of subsets speeds up convergence in terms of reduction of objective function value over time. Superiorization reduces this convergence speed, but the fastest methods without superiorization give rise to the fastest superiorized methods as a rule of thumb.

   The graphic on the right of Figure~\ref{fig:ssaem_kl_time} displays plots of the objective function value versus total variation for the iterates of each method. The curves make clear that the superiorized version of the algorihtms always present the better compromise between data adhesion and smoothness when compared to its respective non-superiorized version. Figure~\ref{fig:seedImages} shows how this property translates to better image quality in practice. The lower row of Figure~\ref{fig:seedImages} shows isocontours overlaid on top of details of the images, in order to show that the superiorized image is noticeably smoother, therefore potentially more useful for visualization tasks. If, on the other hand, some apparent detail seems to have been removed from the image, part of these fine details may be reconstruction artifacts but we make no claims in this direction. Notice that these results are in accordance with those obtained using \textsc{saem} and its superiorized versions presented before.

   \section{Conclusions}

   We have studied the possibility of superiorizing incremental algorithms for maximum likelihood tomographic image reconstruction both from the theoretical and practical point of view. Several options for the superiorization were evaluated, all of which showed promising results. Furthermore, two maximum likelihood algorithms were studied, namely \textsc{saem}~\cite{hcc14} and the newly introduced \textsc{ssaem}.

   Provided numerical evidence indicates that the extra computational burden posed by the superiorization approach is justifiable given the observed improvement in image quality. Furthermore, the new algorithm \textsc{ssaem} was shown to have better convergent properties than \textsc{saem} and related methods while retaining its good practical characteristics.

   \section*{Acknowledgements}

   We are grateful to \textsc{lnls} for providing the beamtime for the acquisition of tomographic data for the apple seed reconstruction experiments.

   \bibliography{total.bib}
   \bibliographystyle{helou_en}

\end{document}